\documentclass[onefignum,onetabnum]{siamonline250211}

\usepackage[utf8]{inputenc}
\usepackage{bm}
\usepackage[english]{babel} 
\usepackage{amsmath,amssymb,mathrsfs,mathtools}
\usepackage{subfigure}
\usepackage{tensor}
\usepackage{physics}
\usepackage[nice]{nicefrac}
\usepackage{amsfonts}
\usepackage{graphicx}
\usepackage{algorithm}
\usepackage{algpseudocodex}
\usepackage{amsopn}

\algtext*{EndProcedure}

\newtheorem{remark}{Remark}

\newtheorem{condition}{Condition}

\usepackage{cleveref}
\crefname{align}{Eq.}{Eqs.}
\crefname{equation}{Eq.}{Eqs.}
\crefname{figure}{Fig.}{Figs.}
\crefname{table}{Table}{Tables}
\crefname{theorem}{Theorem}{Theorems}
\crefname{definition}{Definition}{Definitions}
\crefname{condition}{Condition}{Conditions}
\crefname{lemma}{Lemma}{Lemmas}
\crefname{remark}{Remark}{Remarks}
\crefname{assumption}{Assumption}{Assumptions}
\crefname{proof}{Proof}{Proofs}
\crefname{problem}{Problem}{Problems}
\crefname{proposition}{Proposition}{Propositions}
\crefname{section}{Section}{Sections}
\crefname{subsection}{Section}{Sections}
\crefname{subsubsection}{Section}{Sections}

\ifpdf
  \DeclareGraphicsExtensions{.eps,.pdf,.png,.jpg}
\else
  \DeclareGraphicsExtensions{.eps}
\fi

\usepackage{enumitem}
\setlist[enumerate]{leftmargin=.5in}
\setlist[itemize]{leftmargin=.5in}

\newcommand{\R}{\mathbb{R}}

\newcommand{\diff}{\mathrm{d}}

\newcommand{\quat}{\mathrm{quat}} 
\newcommand{\vecH}{\mathrm{vec}_{\HH}} 
\newcommand{\vecIH}{\mathrm{vec}_{\IH}} 

\newcommand{\HH}{\mathbb{H}}
\newcommand{\IH}{\mathbb{IH}}

\newcommand{\Ree}{\operatorname{Re}}
\newcommand{\Ad}[1]{#1^{\!\dagger}} 
\newcommand{\q}[1]{q^{(#1)}}
\newcommand{\qic}[1]{\overline{q^{(#1)}}}
\newcommand{\p}[1]{p^{(#1)}}

\newcommand{\conj}[1]{\overline{#1}}

\newcommand{\defeq}{\coloneqq}

\newcommand{\admsWA}{\mathcal{A}_{\mathrm{WA}}} 
\newcommand{\admsKS}{\mathcal{A}_{\mathrm{KS}}} 
\newcommand{\inftyWA}{\mathcal{S}_{0}} 

\headers{Global Regularization of the Spatial ER3BP}{Kenshiro Oguri}

\title{Global Regularization of the Spatial Elliptic Restricted Three-body Problem: Quaternionic, Hamiltonian Formulation\thanks{Submitted to the editors on \today.
\funding{This material is based upon work supported by the United States Air Force Office of Scientific Research under award number FA9550-25-1-0237.}}}

\author{Kenshiro Oguri\thanks{School of Aeronautics and Astronautics, Purdue University, West Lafayette, IN, USA (\email{koguri@purdue.edu}).}}

\ifpdf
\hypersetup{
  pdftitle={Global Regularization of the Spatial Elliptic Restricted Three-body Problem: Quaternionic, Hamiltonian Formulation},
  pdfauthor={K. Oguri}
}
\fi

\begin{document}

\maketitle

\begin{abstract}
Any objects in space may encounter singularity due to the inverse-square law of the gravitational force, hindering the analysis and propagation of particle motion in space.
\textit{Regularization} in celestial mechanics removes such gravitational singularities in equations of motion by transforming the time and coordinates.
This paper addresses a class of regularization methods known as \textit{global regularization} for restricted three-body problems (R3BPs), which simultaneously removes the two singularities (e.g., Earth and Moon in cislunar space).
It enables analysis, propagation, and prediction of particle motion throughout the domain of R3BPs, which is in stark contrast to regularizing each singularity individually.
In particular, this paper develops a formulation for global regularization of the \textit{spatial} elliptic R3BP (ER3BP) under external perturbations.
The presented formulation leverages mathematical tools from quaternionic (hyper-complex) algebra and Hamiltonian mechanics.
This approach enables us to
(i) elegantly handle the four-dimensional coordinates arising from the dimension-lifting operation for spatial regularization,
(ii) systematically derive the coordinate-momentum transformation while preserving the canonical structure and incorporating perturbing forces,
and 
(iii) provide geometric insights for analyzing the regularity of the transformation, resulting Hamiltonian, and dynamical flow.
Numerical examples demonstrate the validity of the mathematical results as well as significant numerical benefits of the regularization formulation in perturbed Earth-Moon ER3BP.
\end{abstract}



\section{Introduction}
\label{sec:intro}

Newton's law of universal gravitation implies that any objects in space may encounter singularity due to the inverse-square law of the gravity.
Mathematically, as the distance between an object and a gravitational body, $r$, approaches zero, the singularity occurs due to the gravitational force being proportional to $1/r^2$.
Near such singularities, accuracy of numerical solutions to ordinary differential equations (ODEs) severely degrades.
If adaptive-step ODE solvers with error control are used, they may not converge with a reasonable computational time or a step size, failing to complete the numerical integration.

Naturally, studies on mitigating singularities from the equations of motion (EOM) have occupied many scholars in the fields of celestial mechanics, astronomy, and astrodynamics.
Early contributions include those by Thiele \cite{Thiele1895}, Levi-Civita \cite{Levi-Civita1904}, Birkhoff  \cite{Birkhoff1915}, Lema\^itre \cite{Lemaitre1955}, Deprit \cite{Deprit1963}, Szebehely \cite{Szebehely1964}, Broucke \cite{Broucke1965}, Kustaanheimo \cite{Kustaanheimo1965a}, Burdet \cite{Burdet1967}, and Stiefel \cite{Stiefel1971}.
In celestial mechanics and astronomy, 
regularization techniques play a crucial role for enabling efficient, reliable numerical solutions to general (i.e., non-restricted) $N$-body simulations, which has a myriad of applications including the study of binary stars and galaxies, among others \cite{Aarseth1971,Bettis1971,Heggie1974,Mikkola1985,Heggie1993,Aarseth1999,Mikkola2008,Roa2016a}.

The theory of regularization has also attracted much attention in the astrodynamics community.
Studies in astrodynamics consider the motion of artificial satellites in gravitational fields and hence mostly focus on \textit{restricted} problems since it is generally safe to assume that the satellite mass is negligible compared to those of celestial bodies.
Regularization has been successfully applied to various problems, such as trajectory design \cite{Jezewski1977,Engels1981,Thorne1997,Quarta2010,Oguri2024b}, orbit control \cite{Hernandez2016,Hernandez2014}, uncertainty propagation \cite{Junkins2003,Masat2024}, and object tracking \cite{LeGrand2023}.
For example, regularization techniques can substantially improve the numerical performance of trajectory design, as it often involves thousands of repeated numerical ODE integrations;
the algorithm fails or struggles to converge when the trajectory is near singularity, which is not rare in multi-body environments due to the relative ease for spacecraft to travel between gravitational bodies \cite{Russell2006,Campagnola2018a,Sidhoum2024a,Sidhoum2024b,Kumagai2025}.

The literature on theory and techniques for regularization is vast and has a rich history.
To motivate the work presented in this paper, let us introduce several key viewpoints about the existing results on this topic, summarized as follows.
\begin{itemize}[leftmargin=*]
\item 
\textit{\textbf{Linearization and regularization}} are different concepts in general.
Dating back in the 18th century, the work by Laplace \cite{Deprit1994} is one of the initial results that \textit{linearize} the Kepler problem without approximation.
However, this approach does not remove the gravitational singularity, i.e., not regularization.
On the other hand, not all regularization techniques \textit{linearize} the EOM.
Nevertheless, some transformations are known to regularize \textit{and} linearize the Kepler problem, such as those by Levi-Civita \cite{Levi-Civita1904}, by Kustaanheimo and Stiefel (KS) \cite{Kustaanheimo1965a,Stiefel1971}, by Burdet and Ferra\'ndiz \cite{Burdet1967,Ferrandiz1987}, and by Deprit, Elipe, and Ferrer \cite{Deprit1994}.

\item 
\textit{\textbf{Local and global regularization}}:
One of the first regularization techniques is due to Thiele in 1895 \cite{Thiele1895}.
Thiele's study regularizes the planar circular restricted three-body problem (p-CR3BP) with the two primaries of the same mass, later extended to arbitrary mass ratios by Burrau in 1906 \cite{Burrau1906}.
Notably, their formulations simultaneously regularize collisions with both primaries in the p-CR3BP, and hence often called \textit{global} regularization, as introduced by Birkhoff \cite{Birkhoff1915} and popularized by Szebehely \cite{Szebehely1967a};
this paper adopts this naming convention.%
\footnote{This should be distinguished from Heggie's \textit{global} regularization of the non-restricted $N$-body problem \cite{Heggie1974}, which introduces relative states for all particle pairs and regularizes all of them.
While such generality can be valuable in dynamical astronomy, it is less natural for the R3BP setting, in which the primary--primary motion is prescribed and only two collision singularities are relevant.}
It is in stark contrast that many well-known regularization methods such as Levi-Civita, KS, and Burdet transformations \cite{Levi-Civita1904,Kustaanheimo1965a,Stiefel1971,Burdet1967,Ferrandiz1987} remove only one singularity at once, hence called \textit{local} regularization;
note, however, that one may still find local regularization useful in R3BPs and beyond (e.g., R4BPs) when one is only concerned with particle motion around one of the primaries (e.g., trajectories around the Moon in the Sun-Earth-Moon system) \cite{Chelnokov2017,Belbruno2023,Rossi2024,Oguri2024b}.
Bringing the attention back to global regularization, researchers have extended the aforementioned early studies by Thiele, Burrau, and Birkhoff with different coordinates, including the generalization to the \textit{elliptic} R3BP (ER3BP) \cite{Birkhoff1915,Lemaitre1955,Deprit1963,Szebehely1964,Szebehely1964a,Broucke1965,Waldvogel1972}, although these results are all limited to planar problems.
The focus of this paper is \textit{global} regularization of the \textit{spatial} ER3BP.

\item 
\textit{\textbf{Planar to spatial problems}}:
A breakthrough occurred in the 1960s.
The KS transformation \cite{Kustaanheimo1965a,Stiefel1971} enabled regularization of \textit{spatial} problems by lifting the three-dimensional (3-D) Kepler problem to a 4-D linear oscillator.
It opened the door to many applications in celestial mechanics, astronomy, and astrodynamics by allowing the study of particle motion around a gravitational body in the 3-D physical space without singularity \cite{Jezewski1977,Engels1981,Thorne1997,Junkins2003,Quarta2010,Hernandez2016,Masat2024}.
While the KS regularization itself is \textit{local} (removing only one singularity), 
soon after its discovery, Waldvogel and Arenstorf separately developed global regularization methods for the spatial ER3BP in 1967 and 1969, respectively \cite{Waldvogel1967,Arenstorf1969}, extending the Birkhoff's coordinates to the spatial problem.
This paper calls it \textit{Waldvogel-Arenstorf (WA) transformation} for the significance of their contributions.

\item 
\textit{\textbf{Real and quaternionic representations}} can both describe spatial regularization.
The influential monograph on the KS transformation by Stiefel and Scheifele \cite{Stiefel1971} adopts the real-coordinate form to formulate their transformation and associated EOM. 
On the other hand, increasing number of studies opt to represent the KS transformation in a hyper-complex space with three imaginary units (i.e., quaternions) \cite{Velte1978,Vivarelli1983,Deprit1994,Waldvogel2006,Waldvogel2006b,Waldvogel2008,Breiter2017,Breiter2019} because the transformation ``is much easier to interpret and generalize in the language of quaternions'' \cite{Breiter2019}.
Similarly, Broucke \cite{Broucke1965} also discusses the benefit of viewing \textit{planar} regularization through the lens of conformal mapping on the complex plane and presents a formulation that unifies various regularizing coordinates for the \textit{planar} ER3BP.
Naturally, the quaternionic representation has been adopted in some studies for global regularization of the spatial R3BPs as well \cite{Arenstorf1969,Chelnokov2018}.

\item 
\textit{\textbf{Newtonian and Hamiltonian formulations}}:
Another key viewpoint lies in the approach to formulating the transformations.
Early regularization studies mostly rely on the Newtonian (or vectorial) approach, applying coordinate-time transformations at the differential equation level \cite{Thiele1895,Levi-Civita1904,Kustaanheimo1965a}.
In contrast, the Hamiltonian viewpoint reformulates these transformations in phase space \cite{Moser1970,Saha2009,Zhao2015,Meer2021}.
The Hamiltonian formulation is instrumental to systematically derive momenta from regularizing coordinates and incorporate different time-regularization schemes while preserving the canonical structure.
It also facilitates mathematical proof and analysis of underlying dynamical properties such as regularity by providing geometric insights and compact representations of dynamical systems.

\end{itemize}

The present paper builds on and extends this rich literature on regularization, with the focus on \textit{global, spatial ER3BP regularization}.
As discussed above, \cite{Waldvogel1967,Arenstorf1969} achieve global regularization of the spatial ER3BP via the WA transformation.
The formulations presented in \cite{Waldvogel1967,Arenstorf1969} use Lagrangian mechanics formalism and do not explicitly incorporate arbitrary external forces, which can be a limitation for some practical applications.
On the other hand, the present paper provides an alternative derivation of the WA transformation through the lens of the \textit{quaternionic, Hamiltonian} formulation.
This quaternionic, Hamiltonian approach facilitates the mathematical proof and analysis of important properties of the WA transformation and resulting Hamiltonian and differential equations.

This paper adopts the quaternionic formulation for its ability to streamline the derivation and compact expressions of the regularizing transformation.
The quaternionic approach provides a unifying view of the apparently distinct coordinate transformations through the lens of conformal mapping.
Specifically, this paper defines the regularizing coordinates for the spatial ER3BP as a series of conformal(-like) mappings%
\footnote{called ``conformal-\textit{like}'' here because, strictly speaking, the KS transformation is not a conformal mapping in an ordinary sense; see \cref{remark:WA-non-conformal}.}
in the quaternion space, drawing an inspiration from \cite{Waldvogel2008}.

This paper combines the benefits offered by this quaternionic representation with geometric insights provided by the Hamiltonian formalism.
This approach allows us to devise quaternionic cotangent lifting, which systematically derives the momenta mappings associated with these conformal-like mappings in the quaternion space while preserving the canonical structure.
This paper then analyzes the regularity of the WA-transformed Hamiltonian and its dynamical flow by taking advantage of tools from quaternion algebra and Hamiltonian mechanics.
This analytical mechanics approach also makes it straightforward to incorporate the effect of external acceleration, such as solar radiation pressure (SRP), third-body perturbations, and continuous thrusting, into the system through canonical transformations of generalized forces.

The rest of this paper is structured as follows.
\cref{sec:quaternion} provides key concepts and properties of quaternion algebra.
\cref{sec:ER3BP} defines the dynamical system and introduces its Hamiltonian in the shifted coordinates with time transformation;
it also points out that the time transformation alone does not regularize the system.
\cref{sec:transformation} presents the WA coordinates for global regularization of the spatial ER3BP, derives their conjugate momenta using the Hamiltonian formalism, and discusses their key properties.
\cref{sec:regularized_hamiltonian} gives an explicit form of the resulting Hamiltonian in the WA space, proves its regularity, and discusses some numerical implementation aspects.
\cref{sec:examples} numerically demonstrates the validity and effectiveness of the presented regularization formulation in the perturbed Earth-Moon ER3BP.
\cref{sec:conclusions} concludes the paper.

\section{Hyper-complex numbers}
\label{sec:quaternion}

This work extensively uses hyper-complex numbers, in particular quaternions.%
\footnote{\textit{Quaternion} refers to the mathematical object that extends the complex number by introducing three imaginary units, and must not be confused with the \textit{orientation quaternion}, which is a specific instance where \textit{unit quaternion} becomes useful.
Quaternions considered in this work are generally not of unit-norm.}
Denote the set of quaternions by $\mathbb{H}$.
In general, we can express a quaternion $h \in \mathbb{H} $ as follows:
\begin{align}
h = h_1 i + h_2 j + h_3 k + h_4, \quad h_1, h_2, h_3, h_4 \in \R
\end{align}
where $i, j, k$ are the three imaginary units and satisfy the following cyclic property:
\begin{align}
i^2 = j^2 = k^2 = ijk = -1, \quad
ij = -ji = k, \quad
jk = -kj = i, \quad
ki = -ik = j
\end{align}
This implies that quaternion multiplication is not commutative, i.e., $gh \neq hg$ for $g, h \in \mathbb{H} $ in general.
Similarly, we must distinguish divisions from left and right;
we write $h^{-1}g$ or $g h^{-1} $ instead of $\frac{g}{h} $ to distinguish the left/right divisions.
The conjugate of $h\in\mathbb{H} $, denoted by $\conj{h} $, is defined as:
\begin{align}
\conj{h} = - h_1 i - h_2 j - h_3 k + h_4
\end{align}
which implies $\conj{(\conj{h})} =h $, and $\conj{(gh)} = \conj{h}\conj{g} $ for $g, h \in \mathbb{H} $.
The absolute value of $h\in\mathbb{H} $ is given by
\begin{align}
\abs{h} = \abs{\conj{h}} = \sqrt{h\conj{h}} = \sqrt{\conj{h}h} = \sqrt{h_1^2 + h_2^2 + h_3^2 + h_4^2}
\label{eq:quaternion-absolute}
\end{align}
Similar to complex numbers, the absolute value of quaternions is multiplicative:
\begin{align}
\abs{qh} = \sqrt{qh\conj{(qh)}} = \sqrt{qh\conj{h}\conj{q}} 
= \abs{q}\abs{h} = \abs{h}\abs{q}
\quad
\mathrm{for}\ q,h\in \mathbb{H}
\label{eq:quatAbsMultiplicative}
\end{align}
Quaternionic inverse and its conjugate are calculated as:
\begin{align}
h^{-1} = \frac{\conj{h}}{\abs{h}^2},
\quad
\conj{(h^{-1})} = \frac{h}{\abs{h}^2} = (\conj{h})^{-1}
\label{eq:quaternion-inverse}
\end{align}

We define a subset of the quaternion space:
\begin{align}
\begin{aligned}
\IH \defeq \{h \in \mathbb{H} \mid \Ree{(h)} = 0 \},
\label{eq:set-IH}
\end{aligned}
\end{align}
which is useful to relate a quaternion with a 3-D physical quantity (see \cref{remark:IH-and-R3}).

We introduce three maps, $\quat: \R^4\mapsto \HH$, $\vecH: \HH \mapsto \R^4$, and $\vecIH: \IH \mapsto \R^3$ to clarify the relation between a real vector and a quaternion.
Specifically, $\quat{(\cdot)}$ and $\vecH{(\cdot)}$ represent the identification of a real 4-vector as a quaternion and its inverse, respectively, while $\vecIH{(\cdot)}$ identifies a quaternion in $\IH$ as a real 3-vector.
Mathematically, these are defined as:
\begin{align}
\begin{aligned}
\quat(\bm{q}) = q,\quad
\vecH(q) = \bm{q},\quad
\vecIH(p) = \bm{p}
\end{aligned}
\label{eq:quaternion-map}
\end{align}
where $\bm{q} = [q_1, q_2, q_3, q_4]^\top\in\R^4$, $q = q_1 i + q_2 j + q_3 k + q_4\in\HH$, 
$\bm{p} = [p_1, p_2, p_3]^\top\in\R^3$, and $p = p_1 i + p_2 j + p_3 k \in\IH$.

\section{Elliptic Restricted Three-body Problem (ER3BP)}
\label{sec:ER3BP}

\subsection{Formulation in the Pulsating-Rotating Frame}

To model the ER3BP (and the CR3BP as a special case), we adopt the pulsating-rotating frame \cite{Gomez2002a,Park2025}, called $\mathcal{R}$-frame in this paper.
Consider the motion of a massless particle $P_3$ under the gravitational influence of two primaries $P_1$ and $P_2$.
Denote their gravitational parameters by $\mu_1$ and $\mu_2$, where assume $\mu_1 > \mu_2$ without loss of generality.
The frame origin is at the barycenter of $P_1$ and $P_2$.
The basis vectors are denoted by $(\hat{\bm{i}}, \hat{\bm{j}}, \hat{\bm{k}})$, where $\hat{\bm{i}}$ is aligned with the direction from $P_1$ to $P_2$, $\hat{\bm{k}}$ aligned with the angular momentum of the primaries about the barycenter, and $\hat{\bm{j}}$ completes the right-hand system.
Denote quantities expressed in the pulsating-rotating frame by the left superscript ${^{\mathcal{R}}}(\cdot)$, while those in the inertial frame by ${^{\mathcal{N}}}(\cdot)$.

Let us define relevant quantities associated with the pulsating-rotating frame.
Denote the dimensional positions of $P_1$, $P_2$, and $P_3$ with respect to the origin (barycenter) by $\bm{R}_1$, $\bm{R}_2$, and $\bm{R}_3$ (all in $\R^3$), respectively.
Then, $\bm{R}_3$ in the inertial frame ($\mathcal{N}$-frame) is expressed as:
\begin{align}
\begin{aligned}
{^{\mathcal{N}}}\bm{R}_3(T) = \ell(T) C(T)\ {^{\mathcal{R}}}\bm{\rho}_3(t),
\label{eq:inertial-position}
\end{aligned}
\end{align}
where $T$ is the dimensional time, $\ell(T) = \norm{\bm{R}_{12}(T)}_2$ is the instantaneous $P_1$-$P_2$ distance ($\bm{R}_{12} = \bm{R}_2 - \bm{R}_1 $), and $C(T)$ is the rotation matrix relating $\mathcal{R}$-frame and $\mathcal{N}$-frame, given by:
\begin{align}
\begin{aligned}
C(T) &= [{^{\mathcal{N}}}\hat{\bm{i}}(T),\ {^{\mathcal{N}}}\hat{\bm{j}}(T),\ {^{\mathcal{N}}}\hat{\bm{k}}(T)] \\
\hat{\bm{i}} &= \frac{\bm{R}_{12}}{\ell},
\qquad
\hat{\bm{k}} = \frac{\bm{h}}{\norm{\bm{h}}_2},
\qquad
\hat{\bm{j}} = \hat{\bm{k}} \times \hat{\bm{i}},
\qquad
\bm{h} = \bm{R}_{12} \times \frac{\diff \bm{R}_{12}}{\diff T}.
\end{aligned}
\end{align}
Since the $\hat{\bm{i}}$-$\hat{\bm{j}}$ plane contains the orbit of $P_1$ and $P_2$, 
the frame angular velocity is given by
$\bm{\Omega}(T) = \frac{\diff \nu}{\diff T} \hat{\bm{k}}$, where $\nu\in\R$ is the true anomaly of the primary about the barycenter.

Following the standard formulation of R3BPs, we non-dimensionalize the system for canonical formulation and for numerical advantages.
$\bm{\rho}_3$ from \cref{eq:inertial-position} represents the non-dimensional position of $P_3$ in $\mathcal{R}$-frame, and hence $\ell(T)$ serves as the characteristic length.
Denoting the non-dimensional mass parameter and time as $\mu$ and $t$ respectively, these are given by:
\begin{align}
\begin{aligned}
\mu = \frac{\mu_2}{\mu_1 + \mu_2},
\qquad
\frac{\diff t}{\diff T} = \sqrt{\frac{\mu_1 + \mu_2}{\ell(T)^3}}
\end{aligned}
\label{eq:non-dim-param}
\end{align}
The non-dimensional angular velocity of the frame, denoted by $\bm{\omega}(t)$, is given by:
\begin{align}
\bm{\omega}(t) = \omega_z(t) \hat{\bm{k}},
\quad \text{where}\quad
\omega_z(t) = \frac{\diff \nu}{\diff t} = \frac{\diff \nu}{\diff T} \frac{\diff T}{\diff t}
\end{align}
This confirms $\omega_z = 1$ in the CR3BP since $\frac{\diff \nu}{\diff T} = \frac{\diff t}{\diff T} = \sqrt{{(\mu_1 + \mu_2)}/{\ell^3}} $ and $\ell(T)$ is constant ($\because$ the primary orbit is circular).
On the other hand, in the ER3BP, $\omega_z$ is time-varying and given by $\omega_z(t) = \sqrt{1 + e\cos{(\nu(t))}} $, where $e$ is the eccentricity of the primary's orbit.

In the pulsating-rotating frame, the two primaries are fixed at
$\bm{\rho}_1 = -\mu \hat{\bm{i}}$ and $\bm{\rho}_2 = (1 - \mu)\hat{\bm{i}}$.
Let $\bm{r}\in\R^3$ and $\bm{v}\in\R^3$ denote the non-dimensional position and velocity, respectively, of $P_3$ in the pulsating-rotating frame.
They are formally defined as:
\begin{align}
\begin{aligned}
\bm{r}(t) \defeq \bm{\rho}_3(t) = r_x \hat{\bm{i}} +  r_y \hat{\bm{j}} + r_z \hat{\bm{k}}
\qquad
\bm{v}(t) \defeq \frac{{^\mathcal{R}}\diff \bm{r}}{\diff t} = v_x \hat{\bm{i}} +  v_y \hat{\bm{j}} + v_z \hat{\bm{k}}
\end{aligned}
\end{align}
where the left superscript in ${{^\mathcal{R}}\diff(\cdot)}/{\diff t}$ denotes the time differentiation of a quantity as seen by the $\mathcal{R}$ frame.
With this approach, the time varying effects of the frame rotation and $\ell(T)$  are captured by a single term $\omega_z(t)$, which becomes apparent in the Hamiltonian given in \cref{eq:Hamiltonian-ER3BP}.

\subsection{ER3BP Hamiltonian in Shifted Coordinates}
For global regularization of ER3BP, it is convenient to place the primaries symmetrically at $\pm 1/2$ along the $\hat{\bm{i}}$-axis \cite{Birkhoff1915,Lemaitre1955,Thiele1895,Broucke1965,Oguri2024b}.
We apply an affine shift of coordinates within the pulsating-rotating frame.
Denoting the P1/P2 position in the shifted coordinates by $\bm{\rho}_{1,\mathrm{shift}} = -\frac{1}{2} \hat{\bm{i}}$ and $\bm{\rho}_{2,\mathrm{shift}} = \frac{1}{2}\hat{\bm{i}}$, it is clear that $\bm{\rho}_{i,\mathrm{shift}} = \bm{\rho}_{i} - \left(\frac{1}{2} - \mu\right)\hat{\bm{i}}$ for $i=1,2$.
We apply the same affine positional shift to $P_3$:
\begin{align}
\begin{aligned}
\bm{r}_{\mathrm{shift}} \defeq \bm{r} - l_0\hat{\bm{i}},
\quad
\bm{v}_{\mathrm{shift}} \defeq \frac{\diff \bm{r}_{\mathrm{shift}}}{\diff t} = \bm{v},
\quad \text{where}\quad l_0 = \frac{1}{2} - \mu
\end{aligned}
\end{align}

We define the ER3BP Hamiltonian using the shifted coordinates.
First, introduce generalized coordinates $\bm{q} = q_1\hat{\bm{i}} + q_2 \hat{\bm{j}} + q_3 \hat{\bm{k}} $ and associated momenta $\bm{p} = p_1 \hat{\bm{i}} + p_2 \hat{\bm{j}} + p_3 \hat{\bm{k}}$ as:
\begin{align}
\begin{aligned}
\bm{q} = \bm{r}_{\mathrm{shift}} = \bm{r} - l_0\hat{\bm{i}}
,\quad
\bm{p} 
= \bm{v}_{\mathrm{shift}} + \omega_z(t) \hat{\bm{k}} \times \bm{r}_{\mathrm{shift}} 
= \bm{v} + \omega_z(t) \hat{\bm{k}} \times (\bm{r} - l_0\hat{\bm{i}})
\label{eq:qp-original}
\end{aligned}
\end{align}
Now, express the ER3BP Hamiltonian in terms of $(\bm{q}, \bm{p})$ as follows \cite{Birkhoff1915,Lemaitre1955,Thiele1895,Broucke1965,Oguri2024b}:
\begin{align}
\begin{aligned}
H(t, \bm{q}, \bm{p}) = 
\frac{1}{2}\bm{p}\cdot \bm{p} + \omega_z(t)(q_2 p_1 - q_1p_2) + V(\bm{q}) - l_0 q_1,
\ 
V(\bm{q}) = - \left(\frac{1 - \mu}{r_{13}(\bm{q})} + \frac{\mu}{r_{23}(\bm{q})} \right),
\label{eq:Hamiltonian-ER3BP}
\end{aligned}
\end{align}
where $r_{13}$ denotes the $P_1$-$P_3$ distance while $r_{23}$ the $P_2$-$P_3$ distance, given by
\begin{align}
r_{13}(\bm{q}) = \sqrt{\left(q_1 + \frac{1}{2}\right)^2 + q_2^2 + q_3^2},
\quad
r_{23}(\bm{q}) = \sqrt{\left(q_1 - \frac{1}{2}\right)^2 + q_2^2 + q_3^2}.
\end{align}

Since $H$ in \cref{eq:Hamiltonian-ER3BP} is not time-autonomous, we extend the phase space by introducing $q_0 = t$ and its conjugate momentum $p_0 = -H$.
Denoting the Hamiltonian with this extended phase space by $H_e$, it is given by (e.g.,  Chapter 5, \cite{Abraham1978}):
\begin{align}
\begin{aligned}
H_e(q_0, p_0, \bm{q}, \bm{p}) &= H(q_0, \bm{q}, \bm{p}) + p_0 
\label{eq:extendedHamiltonian}
\end{aligned}
\end{align}
where the initial condition of $(q_0, p_0)$ is given by
\begin{align}
q_0(t=0) = 0,\quad
p_0(t=0) = -H(0, \bm{q}(t=0), \bm{p}(t=0))
\label{eq:q0p0-IC}
\end{align}
Applying the Hamilton's equation to $H_e$ yields:
\begin{align}
\begin{aligned}
\frac{\diff \bm{q}}{\diff t} 
=& \frac{\partial H_e}{\partial \bm{p}}
,\qquad 
\frac{\diff \bm{p}}{\diff t} 
= - \frac{\partial H_e}{\partial \bm{q}} + \bm{g}
,\qquad 
\frac{\diff q_0}{\diff t} 
= \frac{\partial H_e}{\partial p_0}
,\qquad 
\frac{\diff p_0}{\diff t} 
= - \frac{\partial H_e}{\partial q_0} + g_0
\label{eq:HamiltonianEOM-original}
\end{aligned}
\end{align}
where $\bm{g}\in\R^3$ and $g_0\in\R$ denote the generalized forces induced by external acceleration $\bm{f}\in\R^3$ (e.g., SRP, continuous thrust) acting on the physical system, given by:
\begin{align}
\bm{g} = \sum_{i=1}^{3} \frac{\partial r_i}{\partial \bm{q}} f_i = \bm{f} ,
\quad
g_0 = - \sum_{i=1}^{3} \frac{\partial H}{\partial p_i}g_i = - \bm{b}(\bm{q},\bm{p}) \cdot \bm{f}
\label{eq:generalized-force-original}
\end{align}
where $\bm{b}: \R \times \R^3\times \R^3 \mapsto \R^3 $ is: 
\begin{align}
\bm{b}(q_0, \bm{q},\bm{p}) \defeq \nabla_{\bm{p}} H = \bm{p} + \omega_z(q_0) q_2 \hat{\bm{i}} - \omega_z(q_0) q_1 \hat{\bm{j}}
\label{eq:b-vector}
\end{align}
The external perturbations $\bm{f}$ is assumed to be regular on the domain of interest.
We can confirm that 
\cref{eq:HamiltonianEOM-original} yields $\diff q_0 / \diff t = 1$, and with \cref{eq:q0p0-IC}, that $q_0(t) = t$.
Similarly,
\cref{eq:HamiltonianEOM-original} together with \cref{eq:generalized-force-original} gives $\diff p_0 / \diff t = - \partial H / \partial t + g_0 = - \diff H / \diff t $, and with \cref{eq:q0p0-IC}, that $p_0(t) = -H$ along the dynamical flow governed by \cref{eq:HamiltonianEOM-original}, where $H$ is generally not constant.

\subsection{Time-transformed ER3BP Hamiltonian}
\label{sec:time-transform}

One of the necessary steps in regularization is time transformation (also called time regularization).
Time transformation has the effect of ``slowing down'' the motion of the particle near a gravitational body.

To perform time transformation, we introduce fictitious time, $s$.
Since we have two singularities in ER3BP, we introduce a fictitious time similar to the one by Thiele \cite{Thiele1895}:
\begin{align}
\begin{aligned}
\frac{\diff t}{\diff s} = \eta(\bm{q}) \defeq r_{13}(\bm{q})\cdot r_{23}(\bm{q}) = 
\sqrt{\left[\left(q_1 + \frac{1}{2}\right)^2 + q_2^2 + q_3^2\right] \left[\left(q_1 - \frac{1}{2}\right)^2 + q_2^2 + q_3^2\right]}
\label{eq:fictitious-time}
\end{aligned}
\end{align}
To preserve the canonical structure under this time transformation, we introduce the time-regularized Hamiltonian:
\begin{align}
\begin{aligned}
\hat{H}_e(q_0, p_0, \bm{q}, \bm{p}) \defeq \eta(\bm{q}) \cdot H_e(q_0, p_0, \bm{q}, \bm{p})
\label{eq:timeHamiltonian}
\end{aligned}
\end{align}
where recall $H_e$ is defined in \cref{eq:extendedHamiltonian}.
Similarly, we transform the generalized forces to establish a new canonical system that represents the equivalent physical system, as follows:
\begin{align}
\begin{aligned}
\hat{\bm{g}} 
= \eta\, \bm{g} 
= \eta\, \bm{f} 
,\qquad
\hat{g}_0  
= \eta\, g_0
= - \eta\, \bm{b}\cdot \bm{f}
\label{eq:generalized-force-time}
\end{aligned}
\end{align}
where recall \cref{eq:b-vector,eq:fictitious-time} for $\bm{b}$ and $\eta$.
The equations of motion are given by:
\begin{align}
\begin{aligned}
\frac{\diff \bm{q}}{\diff s} 
=& \frac{\partial \hat{H}_e}{\partial \bm{p}}
,\qquad
\frac{\diff \bm{p}}{\diff s} 
= - \frac{\partial \hat{H}_e}{\partial \bm{q}} + \hat{\bm{g}}
\qquad
\frac{\diff q_0}{\diff s} 
= \frac{\partial \hat{H}_e}{\partial p_0}
,\qquad
\frac{\diff p_0}{\diff s} 
= - \frac{\partial \hat{H}_e}{\partial q_0} + \hat{g}_0
\label{eq:HamiltonianEOM-time}
\end{aligned}
\end{align}
We can again confirm that combining \cref{eq:HamiltonianEOM-time,eq:generalized-force-time} with the initial condition \cref{eq:q0p0-IC} implies that $q_0 = t$ and $p_0 = - H$ along the dynamical flow.

Here, we can also see that singularities remain in \cref{eq:HamiltonianEOM-time}, confirming that applying only the time regularization is not sufficient to regularize the ER3BP or CR3BP. 
Due to the presence of the square root in $r_{13}$ and $r_{23}$, a partial derivative of them with respect to $\bm{q}$ yields division by $r_{13}$ or $r_{23}$, respectively, which leads to zero-division when $P_3$ collides with $P_1$ or $P_2$.

\section{Spatial Extension of Birkhoff's Coordinates}
\label{sec:transformation}

\subsection{Waldvogel-Arenstorf Coordinates}
\label{sec:WA-coordinates}

As proved by Hurwitz, Levi-Civita-like transformations exist in spaces of dimension $m = 1,2,4,8$ \cite{Roa2016a}.
This implies that regularization of spatial problems must lift $\R^3$ to $\R^4$, as is done in the KS transformation.
We use a transformation $\bm{z} \mapsto \bm{q}$ where $\bm{q}\in\R^3$ is defined in \cref{eq:qp-original} while $\bm{z}\in\R^4 $ is the lifted coordinates.
We use \cref{eq:quaternion-map} to identify $\bm{z} $ and $ \bm{q}$ as quaternions:
\begin{align}
\begin{aligned}
z = \quat(\bm{z}) = z_1 i + z_2 j + z_3 k + z_4,
\quad
q = \quat(\bm{q}) = q_1 i + q_2 j + q_3 k + q_4 \quad (q_4 = 0)
\end{aligned}
\end{align}
where note that $q\in\IH \subset \HH $ since $q_4=0$; recall the definition of $\IH$ in \cref{eq:set-IH}.

\begin{remark}\label{remark:IH-and-R3}
We can regard a subspace $\IH \times \IH $ as a quaternionic representation of the physical space $\R^3\times \R^3$.
For coordinates and momenta in the physical space $\bm{q} = q_1 \hat{\bm{i}} + q_2 \hat{\bm{j}} + q_3 \hat{\bm{k}} $ and $\bm{p} = p_1 \hat{\bm{i}} + p_2 \hat{\bm{j}} + p_3 \hat{\bm{k}} $, their quaternionic representations live in $\IH $ as $q = q_1 i + q_2 j + q_3 k \in \IH $ and $p = p_1 i + p_2 j + p_3 k \in \IH $, indicating the correspondence of $(q_1, q_2, q_3)$ and $(p_1, p_2, p_3)$ in the physical space and the quaternionic space.
\end{remark}

The Waldvogel-Arenstorf (WA) map \cite{Waldvogel1967,Arenstorf1969} combines the Birkhoff's coordinates (\textit{planar, global} regularization) and the KS coordinates (\textit{spatial, local} regularization) for \textit{spatial, global} regularization.
The underlying idea is as follows.
As discussed in \cite{Waldvogel1967,Waldvogel2008}, one can regard the Birkhoff's mapping as a composition of two M\"obius maps in the complex space and a Levi-Civita map.
This observation implies that one can extend the Birkhoff's coordinates to spatial problems by
1) introducing the two M\"obius maps in the \textit{quaternionic} space,
and
2) replacing the Levi-Civita map with the KS map as a quaternionic extension.

Thus, the WA map consists of three coordinate transformations.
Letting $\q{0}(\defeq q)$, $\q{1}$, $\q{2}$, $\q{3}(\defeq z) $ represent the intermediate coordinates, the WA map for coordinates, denoted by $\mathcal{WA}_q$, is expressed as:
\begin{align}
\mathcal{WA}_q:\ \HH \mapsto \HH,\qquad z \mapsto q,
\quad \mathrm{with} \quad
z = \q{3} \xmapsto{\mathcal{M}^{(1)}_q} \q{2} \xmapsto{\mathcal{KS}_q} \q{1} \xmapsto{\mathcal{M}^{(2)}_q} \q{0} = q
\label{eq:WA-coordinates}
\end{align}
where $\mathcal{M}^{(1)}_q$ and $\mathcal{M}^{(2)}_q$ denote the two M\"obius maps, and $\mathcal{KS}_q$ the coordinate part of the quaternionic KS map.
These maps are given by:%
\footnote{Note a notational difference from \cite{Waldvogel2008}, which uses the definition $z = q_1 + q_2 i + q_3 j + q_4 k$ with $q_4 = 0$, so the basis vectors of the physical space $(\hat{\bm{i}}, \hat{\bm{j}},\hat{\bm{k}})$ are aligned with the real axis, $i$-axis, and $j$-axis, respectively.
On the other hand, this paper takes $(\hat{\bm{i}}, \hat{\bm{j}},\hat{\bm{k}})$ to be aligned with the imaginary $i$-axis, $j$-axis, and $k$-axis, respectively, for easier correspondence.}
\begin{subequations}
\label{eq:mappingComposition}
\begin{align}
\mathcal{M}^{(1)}_q:\ & \HH \mapsto \HH, \quad&&
\q{2} \defeq \mathcal{M}^{(1)}_q(\q{3}) = i - {2}{(\q{3} - i)^{-1}} &&
\quad (\text{M\"obius map})
\label{eq:mobius-a}
\\
\mathcal{KS}_q\ :\ & \HH \mapsto \HH, \quad&&
\q{1} \defeq \mathcal{KS}_q(\q{2}) = \qic{2}i\q{2} &&
\quad (\text{KS map})
\label{eq:hopf-b}
\\
\mathcal{M}^{(2)}_q:\ & \HH \mapsto \HH &&
\q{0} \defeq \mathcal{M}^{(2)}_q(\q{1}) = \frac{1}{2}\left[i - {2}{(\q{1} - i)^{-1}} \right] &&
\quad (\text{M\"obius map})
\label{eq:mobius-c}
\end{align}
\end{subequations}%

This approach leverages the two facts summarized in \cref{theorem:Hopf_Mobius}.

\begin{lemma}\label{theorem:Hopf_Mobius}
{\upshape (i)} the KS map $y = \conj{x}ix $ sends any quaternion $x\in\mathbb{H}$ to a pure-imaginary quaternion, i.e., $y\in\IH $.
{\upshape (ii)} the M\"obius map $y = i - {2}{(x - i)^{-1}}$ preserves the pure imaginary property, i.e., if $x\in\IH$, then $y \in \IH$ under $y = i - {2}{(x - i)^{-1}}$.
\end{lemma}
\begin{proof}
(i) $\conj{y} = \conj{(\conj{x}ix)} = \conj{x}\conj{i}x = - \conj{x}ix = - y $, which implies $\Ree(\conj{y}) = 0$.
(ii) Suppose $x \in\IH$.
Then, $\conj{x} = - x$, and hence $(x - i)^{-1} = \frac{- x + i}{\abs{x - i}^2}  $.
It implies
$y = i - {2}{(x - i)^{-1}} = i + \frac{2}{\abs{x - i}^2}x - \frac{2}{\abs{x - i}^2}i $, which is a real linear combination of $x$ and $i$, and hence pure imaginary.
\end{proof}

\begin{remark}\label{remark:WA-non-conformal}
While the WA map \cref{eq:mappingComposition} can be viewed as a lifted spatial analogue of the Birkhoff's map, strictly speaking, it is not conformal (as opposed to the Birkhoff's map).
In the planar problem, complex conformal maps reshape the configuration space while preserving the metric up to a scalar factor \cite{Thiele1895,Broucke1965}.
In three dimensions, conformal maps are more restrictive and cannot provide the same kind of regularization.
The WA construction therefore lifts the spatial problem through a KS map.
However, the KS map is not conformal;
since motion along the redundant direction in the lifted KS space does not change the physical position, it acts conformally only on the 3-D subspace orthogonal to the redundant direction.%
\footnote{Note that the discussion here is consistent with the property Kustaanheimo and Stiefel describe in \cite{Kustaanheimo1965a} about the KS mapping as: ``as conformal as a mapping of an $\R^4$ onto an $\R^3$ ever can be.''}
Regardless, after accounting for the redundant direction, \cref{eq:mappingComposition} retains the scalar metric-scaling property that makes conformal maps useful for regularization.
\end{remark}

\subsection{Canonical Momenta in Waldvogel-Arenstorf Coordinates}
\label{sec:WA-momenta}

Now, let us derive the momenta for the WA coordinates while preserving the canonical structure.

To derive canonical momenta,
this paper takes the cotangent-lift approach instead of the generating-function formulation.
The cotangent-lift approach appears more natural for quaternionic formulations.
As discussed in \cref{sec:quaternion}, the quaternionic algebra is not commutative, which makes the differentiation and multiplication ordering cumbersome;
if one were to take the generating function approach, one would have to properly take care of the non-commutative property of quaternions.
In addition, the class of functions that are \textit{quaternion-differentiable} is restrictive compared to real- or complex-valued functions \cite{Sudbery1979}.
On the other hand, the cotangent-lift approach allows us to derive the momenta without requiring the quaternionic differentiability while explicitly preserving the canonical structure.

\subsubsection{Quaternionic Cotangent Lift}
Let us first introduce a general procedure for calculating momenta in quaternion form for a given smooth coordinate map, which is crucial for deriving momentum mapping for each coordinate map in \cref{eq:mappingComposition}.
This paper calls this approach \textit{quaternionic cotangent lift}.
As preparation, let us introduce two quaternionic operators needed.
First, define the operator $\langle x,y\rangle_{\R}$ for $x,y\in \HH$:
\begin{align}
\label{eq:pairing}
\langle x,y\rangle_{\R} := \vecH{(x)} \cdot \vecH{(y)} = \Ree\!\left(x\,\conj{y}\right),
\end{align}
which can be interpreted as the vector inner product with the identification of inputs in $\HH$ as $\R^4$;
for example, if $x=x_1i+x_2j+x_3k+x_4$ and $y=y_1i+y_2j+y_3k+y_4$, then $\langle x,y\rangle_{\R} =x_1y_1+x_2y_2+x_3y_3+x_4y_4$.
Next, define the adjoint of a real-linear map%
\footnote{A map $L:\HH\to\HH$ is called real-linear if $L(\alpha x+\beta y)=\alpha L(x)+\beta L(y)$ for all $x,y\in\HH$ and $\alpha,\beta\in\R$.}
$L$ and denote it by $\Ad{L}$, which satisfies
\begin{align}
\label{eq:adjoint-def}
\langle x, L(y)\rangle_{\R} = \langle \Ad{L}(x), y\rangle_{\R} \quad \forall x,y\in\HH.
\end{align}
We have the following useful identity regarding the adjoint operator for a map $L_{a,b}(y)\defeq a y b$:
\begin{align}
\label{eq:adjoint-rule}
L_{a,b}(y)=a y b 
\quad \iff\quad
\Ad{L_{a,b}}(x)=\bar a\,x\,\bar b
\end{align}
because
$\langle x,L_{a,b}(y)\rangle_{\R}
= \Ree\!\left(x\,\overline{a y b}\right)
= \Ree\!\left(x\,\bar b\,\bar y\,\bar a\right)
= \Ree\!\left(\bar a\,x\,\bar b\,\bar y\right)
= \left\langle \bar a\,x\,\bar b,y\right\rangle_{\R}$.

These operators aid in deriving quaternionic cotangent lift in a general form (\cref{theorem:quat-contangent-lift}) and for the quaternionic M\"obius map (\cref{theorem:mobius-momenta}).

\begin{lemma}\label{theorem:quat-contangent-lift}
The canonical momentum map $p \mapsto P$ associated with a smooth quaternionic coordinate map $Q = F(q)$ is obtained by solving \cref{eq:P-to-p} for $P$:
\begin{align}
p=\Ad{(\diff F_q)}(P)
\label{eq:P-to-p}
\end{align}
where $\Ad{(\cdot)}$ is the adjoint of a map as defined in \cref{eq:adjoint-def}, and $\diff F_q$ is the differential of $F$.
\end{lemma}

\begin{proof}
It is widely known that the real-valued canonical map for coordinates and momenta $(\bm{q},\bm{p})\mapsto (\bm{Q}, \bm{P})$ satisfies the following equation (e.g., \cite[Sec.~6.3]{Marsden1999}):
\begin{align}
\bm{p} = \mathbf{J}(\bm{q})^\top \bm{P},\quad \text{where}\quad \mathbf{J}(\bm{q}) = \frac{\partial \bm{Q}}{\partial \bm{q}}
\end{align}
which implies $\bm{p} ^\top \delta \bm{q} = \bm{P}^\top \mathbf{J}(\bm{q})\delta \bm{q}$ for arbitrary variation $\delta\bm{q} $.
Rewriting it in a variational form and using the original quaternionic coordinates and momenta $({q},{p})$ and their canonical counterpart $({Q}, {P})$, we have
\begin{align}
\label{eq:quaternionic-cotangent-lift}
\Ree\!\left(p\,\overline{\delta q}\right)
=
\Ree\!\left(P\,\overline{\diff F_q(\delta q)}\right)
,\quad \forall\delta q\in\HH
\end{align}
where $\diff F_q:\HH\to\HH$ maps an arbitrary variation $\delta q\in\HH$ to the variation in $Q\in\HH$, i.e., $\delta Q = \diff F_q(\delta q)$.
Like usual cotangent lift, \cref{eq:quaternionic-cotangent-lift} ensures that the preservation of the canonical one-form.
Noting that \cref{eq:quaternionic-cotangent-lift} is equivalent to $\langle p, \delta q\rangle_{\R} = \langle P, \diff F_q(\delta q)\rangle_{\R}$ and using \cref{eq:adjoint-def}, we have
\begin{align}
\langle p, \delta q\rangle_{\R} = \langle P, \diff F_q(\delta q)\rangle_{\R} \underset{\cref{eq:adjoint-def}}{=} \langle \Ad{(\diff F_q)}(P), \delta q\rangle_{\R} \  \forall\,\delta q\quad
\iff\quad
p = \Ad{(\diff F_q)}(P)
\label{eq:cotangent-lift}
\end{align}
which completes the proof.
\end{proof}

\begin{remark}\label{remark:quaternionic-momenta}
In general, \cref{eq:P-to-p} for quaternionic momenta can be calculated as:
\begin{align}
p=\Ad{(\diff F_q)}(P) = \sum_{k=1}^m \overline{a_k(q)}\,P\,\overline{b_k(q)}
\end{align}
where $a_k,b_k \in\HH $ are quaternions that satisfy $\diff F_q(\delta q) = \sum_{k=1}^m a_k\,\delta q\,b_k$ and hence depend on the map $F$ and the reference point $q$.
Such decomposition of $\diff F_q$ into $a_k, b_k$, and $\delta q$ is guaranteed to exist since
any real-linear map can be expressed as a finite sum of left-right multiplication operators (e.g., see \cite[Lemma~7.5.4]{Voight2021}).
\end{remark}

\subsubsection{Quaternionic Waldvogel-Arenstorf Momenta Derivation}

Similar to \cref{eq:WA-coordinates}, 
denoting the intermediate momenta by $\p{0}(\defeq p), \p{1}, \p{2}, \p{3}(\defeq \xi) $, the WA map for momenta, $\mathcal{WA}_p$, is expressed as:
\begin{align}
\begin{aligned}
\mathcal{WA}_p:&\ \HH \times \HH \mapsto \HH,\qquad (z, \xi) \mapsto p,
\\
&\mathrm{with} \quad
(z, \xi) = (\q{3}, \p{3}) \xmapsto{\mathcal{M}^{(1)}_p} (\q{2}, \p{2}) \xmapsto{\mathcal{KS}_p} (\q{1}, \p{1}) \xmapsto{\mathcal{M}^{(2)}_p} \p{0} = p
\end{aligned}
\end{align}
where $\mathcal{KS}_p$ is the momenta part of the quaternionic KS map, while 
$\mathcal{M}^{(1)}_p$ and $\mathcal{M}^{(2)}_p$ are the momenta maps associated with $\mathcal{M}^{(1)}_q$ and $\mathcal{M}^{(2)}_q$ from \cref{eq:mappingComposition}.

Let us first discuss $\mathcal{KS}_p$.
Building on the quaternionic view of the KS coordinates presented in \cite{Velte1978,Vivarelli1983,Deprit1994,Waldvogel2006,Waldvogel2008}, some recent studies \cite{Saha2009,Zhao2015,Breiter2017,Takeuchi2025} derive the associated momenta mapping $\mathcal{KS}_p$ in quaternion form as follows:
\begin{align}
\mathcal{KS}_p: \admsKS \mapsto \IH,\qquad 
\p{1} = \mathcal{KS}_p(\q{2}, \p{2}) = \frac{1}{2|{\q{2}}|^2}\conj{\q{2}} i\p{2}
\label{eq:hopf-cotangent}
\end{align}
where $\admsKS$ is the admissible set for $(\q{2}, \p{2})$ so that the image of $\mathcal{KS}_p$ lies in $\IH$, defined as:
\begin{align}
\begin{aligned}
\admsKS \defeq \left\{(q,p)\in\HH\times\HH \mid \Ree\left(\conj{q}\, i p\right)=0 \right\}
\end{aligned}
\label{eq:KS-admissible-set}
\end{align}
This condition corresponds to the so-called \textit{bilinear} relation in literature in the real-number form (often described as $L(\bm{q})\bm{p} = L(\bm{p})\bm{q} $ where $L(\cdot)$ is the KS matrix; see, for instance, \cite{Stiefel1971}).

\begin{remark}\label{remark:hopf-momenta}
Notice that in \cref{eq:hopf-cotangent}, $\p{1} = \conj{\q{2}} i\p{2} / (2|{\q{2}}|^2) $ does not necessarily belong to $\IH$ for an arbitrary pair $(\q{2} ,\p{2})\in\HH\times\HH$.
Thus, ensuring that $(\q{2}, \p{2})$ belongs to the set $\admsKS$ is crucial for $\mathcal{KS}_p$ to always return momenta in $\IH$, and hence in the 3-D physical space.
This aspect is discussed further and resolved in \cref{sec:WA-transformation}.
\end{remark}

Next, let us consider $\mathcal{M}^{(1)}_p$ and $\mathcal{M}^{(2)}_p$.
Their analytical expressions are available in the quaternionic form in \cite{Takeuchi2025} \textit{without} derivation.%
\footnote{Note that \cite{Takeuchi2025} studies the WA transformation for the problem of two fixed centers, also known as the \textit{Euler's two-center problem}.
It considers the motion of a particle under gravitational attraction of two point masses fixed in the inertial frame (not orbiting each other).
One can regard it as a simplified CR3BP with the Coriolis and centrifugal forces ignored.
See \cite{Lukyanov2005} for a review of the two-center problem.}
Thus, for completeness, \cref{theorem:mobius-momenta} derives $\mathcal{M}^{(1)}_p$ via the cotangent-lift approach.
It is straightforward to adapt \cref{theorem:mobius-momenta} to $\mathcal{M}^{(2)}_p$.

\begin{lemma}
\label{theorem:mobius-momenta}
Consider a M\"obius map for quaternionic coordinates $\q{3}\in\HH$ and $\q{2}\in\HH:$
\begin{align}
\mathcal{M}_q^{(1)} : \HH \mapsto \HH,\quad 
\q{2} = \mathcal{M}_q^{(1)} (\q{3}) = i - 2 (\q{3}-i)^{-1}.
\label{eq:mobius}
\end{align}
Then, the canonical momenta mapping associated with this transformation is given by 
\begin{align}
\mathcal{M}_p^{(1)}: \HH \times \HH \mapsto \HH,\quad 
\p{2} = \mathcal{M}_p^{(1)}(\q{3},\p{3}) = \frac{1}{2}\conj{(\q{3} - i)}\p{3}\conj{(\q{3} - i)}.
\label{eq:Mobius-cotangent}
\end{align}
\end{lemma}
\begin{proof}
In this proof, we rewrite $\mathcal{M}_q^{(1)} , \q{3}, \q{2}, \p{3}, \p{2} $ as $F,q,Q,p,P$ for notational brevity.
First, note that $F = i - 2 (q-i)^{-1} = (i q - 1)\,(q - i)^{-1}$.
Calculating the differential $\diff F_q $ using the product rule and the identity $\diff (X^{-1}) = -X^{-1}(\diff X)X^{-1}$ yields:
\begin{align}
\diff F_q(\delta q) = i\,\delta q\,(q - i)^{-1} - (i q - 1)\,(q - i)^{-1}\,\delta q\,(q - i)^{-1}
\label{eq:mobius-differential}
\end{align}
Using \cref{theorem:quat-contangent-lift} while applying \cref{eq:adjoint-rule} to \cref{eq:mobius-differential} term-by-term, we obtain a form that is consistent with the claim in \cref{remark:quaternionic-momenta}:
\begin{align}
\begin{aligned}
p 
&= \Ad{(\diff F_q)}(P)
= -i\,P\, \conj{(q - i)^{-1}} +  \conj{(q - i)^{-1}}\,( \conj{q}\,i + 1)\,P\, \conj{(q - i)^{-1}}
\\
&= -i\,P\,(\conj{q} + i)^{-1} + (\conj{q} + i)^{-1}\,(\conj{q}\,i + 1)\,P\,(\conj{q} + i)^{-1},
\label{eq:mobius-intermediate}
\end{aligned}
\end{align}
where $\overline{(q - i)^{-1}} = (\conj{q} + i)^{-1}$ using \cref{eq:quaternion-inverse}.
Using $\conj{q}\,i + 1 = (\conj{q} - i)\,i = [(\conj{q} + i) - 2i]\,i$,
\begin{align}
(\conj{q} + i)^{-1}(\conj{q}\,i + 1)
= [1 - 2 (\conj{q} + i)^{-1} i]i
= i + 2(\conj{q} + i)^{-1}.
\end{align}
Substituting this into \cref{eq:mobius-intermediate} yields $p = 2(\conj{q} + i)^{-1}P(\conj{q} + i)^{-1}$;
therefore, the cotangent lift of the M\"obius transformation \cref{eq:mobius} is given by
$p = 2\,(\conj{q} + i)^{-1}\;P\;(\conj{q} + i)^{-1} \iff P = \frac{1}{2}\conj{(q - i)}p\conj{(q - i)}$, completing the proof.
\end{proof}

\subsection{Condition for Physical Projection}
\label{sec:WA-transformation}

Let us now discuss the condition required for the WA transformation to return coordinates and momenta in the three-dimensional physical space, i.e., $(\q{0}, \p{0})\in \IH\times \IH $.

First, summarizing \cref{eq:mappingComposition,eq:hopf-cotangent,eq:Mobius-cotangent}, we have the WA transformation as follows:
\begin{subequations}
\label{eq:mappingCompositionMomenta}
\begin{align}
&(\q{3}, \p{3}) \mapsto (\q{2}, \p{2}) : &&
\q{2} = i - {2}{(\q{3} - i)^{-1}}, &&
\p{2} = \frac{\conj{(\q{3} - i)} \p{3} \conj{(\q{3} - i)}}{2}
\label{eq:mobius-all-a}
\\
&(\q{2}, \p{2}) \mapsto (\q{1}, \p{1}) : &&
\q{1} = \conj{\q{2}}i\q{2},
 &&
\p{1} = \frac{1}{2|{\q{2}}|^2}\conj{\q{2}} i\p{2}
\label{eq:hopf-all-b}
\\
&(\q{1}, \p{1}) \mapsto (\q{0}, \p{0}) : &&
\q{0} = \frac{1}{2}[i - {2}{(\q{1} - i)^{-1}} ],
 &&
\p{0} = \conj{(\q{1} - i)} \p{1} \conj{(\q{1} - i)}
\label{eq:mobius-all-c}
\end{align}
\end{subequations}

First focusing on the coordinates, \cref{theorem:Hopf_Mobius} and \cref{eq:mappingComposition} imply that $\mathcal{KS}_q$ in \cref{eq:mappingComposition} always returns $\q{0}\in\IH$ for any $\q{3}\in\HH $, because $\mathcal{KS}_q$ sends any $\q{2}\in\HH $ to $\q{1}\in\IH $ due to (i) of \cref{theorem:Hopf_Mobius}, and $\mathcal{M}^{(2)}_q$ sends $\IH$ to $\IH$ due to (ii) of \cref{theorem:Hopf_Mobius}.
Thus, \cref{eq:mappingComposition} is guaranteed to project a quaternionic quantity onto physical coordinates.

On the other hand, the momenta mapping requires some attention.
In general, applying \cref{eq:mappingCompositionMomenta} to an arbitrary WA state $(z, \xi)\in\HH\times\HH$ may not produce a purely imaginary momentum, which is necessary for the momentum to be in the 3-D physical space.
As discussed in \cref{remark:hopf-momenta}, ensuring $\p{1} \in \IH$ requires $(\q{2}, \p{2})\in \admsKS $.
Thus, we need to derive the condition for $(\q{3},\p{3}) $ to ensure that $\mathcal{M}_p^{(1)} $ always returns $(\q{2}, \p{2})\in \admsKS $.
To that end, denote the admissible set for the WA transformation by $\admsWA $, defined as:
\begin{align}
\admsWA \defeq \left\{(z, \xi) \in \HH \times \HH \mid \Ree\left[(\conj{z} + i)\xi(\conj{z} - i)\right] = 0 \right\}.
\label{eq:WA-admissibility}
\end{align}
where recall $(\q{3}, \p{3}) = (z, \xi) $ and $(\q{0}, \p{0}) = (q, p) $.
\cref{theorem:WAdmissibleCharacterization} shows that $(\q{3},\p{3}) \in \admsWA $ is equivalent to $(\q{2}, \p{2})\in \admsKS $, and hence to $(q, p) \in \IH \times \IH$ (i.e., the coordinates and momenta are in the 3-D physical space).

\begin{lemma}
\label{theorem:WAdmissibleCharacterization}
Let $(z, \xi) = (\q{3}, \p{3})$ be a WA state for which \cref{eq:mappingCompositionMomenta} is well-defined, and let $(q, p) = (\q{0}, \p{0}) $ be its image under that map.
Then, $(q, p) \in \IH \times \IH  \iff  (z, \xi) \in \admsWA$.
\end{lemma}
\begin{proof}
This proof uses $(\q{0}, \p{0})$ and $(\q{3}, \p{3})$ instead of $(q, p)$ and $(z, \xi)$ for notational clarity.
So, the goal is to show that $(\q{0}, \p{0}) \in \IH \times \IH \iff (\q{3}, \p{3}) \in \admsWA$.
First, $\q{1} = \conj{\q{2}}i\q{2}$ implies $\q{1}\in\IH$ due to (i) of \cref{theorem:Hopf_Mobius}.
Then, the last M\"obius step in \cref{eq:mappingCompositionMomenta} yields $\q{0}\in\IH$ due to (ii) of \cref{theorem:Hopf_Mobius}.
Hence
\begin{align}
\begin{aligned}
(\q{0}, \p{0}) \in \IH \times \IH
 \iff 
\p{0}\in\IH.
\label{eq:LM4-eq1}
\end{aligned}
\end{align}
Calculating the inverse of the second equation of \cref{eq:mobius-all-c} yields
$\p{1} = (\conj{\q{1}} + i)^{-1}\p{0}(\conj{\q{1}} + i)^{-1}$.
Since $\q{1}\in\IH$, both $\conj{(\q{1} - i)} = -\q{1} + i$ and $(\conj{\q{1}} + i)^{-1} = (-\q{1} + i)^{-1}$, which appear in \cref{eq:mobius-all-c} and its inverse, are pure imaginary.
Also, for any $u,v\in\IH$, it holds that $\conj{(uvu)} = \conj{u}\conj{v}\conj{u} = (-u)(-v)(-u) = -uvu$.
Thus, $uvu\in\IH$ whenever $u,v\in\IH$. 
The two relations above imply
\begin{align}
\begin{aligned}
\p{0}\in\IH
 \iff 
\p{1}\in\IH
 \iff 
\Ree(\p{1}) = 0
\label{eq:p0-p1-pureImaginary}
\end{aligned}
\end{align}

Now, from \cref{eq:mobius-all-a}, we have
\begin{align}
&\conj{\q{2}} = - i - 2(\conj{\q{3}} + i)^{-1} = - (\conj{\q{3}} + i)^{-1}(\conj{\q{3}} - i)i 
 \iff 
\conj{\q{3}} - i = (\conj{\q{3}} + i)\conj{\q{2}}i,
\label{eq:LM4-temp1}
\\
&2\p{2} = \conj{(\q{3} - i)} \p{3} \conj{(\q{3} - i)} = (\conj{\q{3}} + i) \p{3} (\conj{\q{3}} + i)
\label{eq:LM4-temp3}
\end{align}
Thus, the defining condition of $\admsWA$ in \cref{eq:WA-admissibility} is equivalent to
\begin{align}
\begin{aligned}
\Ree\left[(\conj{\q{3}} + i)\p{3}(\conj{\q{3}} - i)\right]
&=
\Ree\left[(\conj{\q{3}} + i)\p{3}(\conj{\q{3}} + i)\conj{\q{2}}i\right]
=
2\Ree(\p{2}\conj{\q{2}}i)
\\
&= 2\Ree(\conj{\q{2}}i\p{2}) = 4|\q{2}|^2 \Ree(\p{1})
\label{eq:LM4-temp4}
\end{aligned}
\end{align}
where the first identity is due to \cref{eq:LM4-temp1}, the second is due to \cref{eq:LM4-temp3}, the third is due to invariance of the real part of a quaternion product, and the fourth is due to \cref{eq:hopf-all-b}.
Therefore, \cref{eq:LM4-temp4,eq:WA-admissibility} imply
\begin{align}
(\q{3}, \p{3}) \in \admsWA  \iff  \Ree(\p{1}) = 0.
\label{eq:LM4-temp5}
\end{align}
Combining \cref{eq:LM4-temp5} with \cref{eq:p0-p1-pureImaginary,eq:LM4-eq1} completes the proof.
\end{proof}

\cref{theorem:WAdmissibleCharacterization} clarifies the following two facts:
(i) if we transform a pair of physical coordinates and momenta to $(z,\xi)$ via the inverse of \cref{eq:mappingCompositionMomenta}, then such $(z,\xi)$ will automatically reside in $\admsWA$.
Also, (ii) if we have $(z,\xi)\in\admsWA$, then \cref{eq:mappingCompositionMomenta} always returns $(q,p)\in\IH \times \IH$, which can then be identified as a physical state $(\bm{q},\bm{p})\in \R^3\times\R^3$ via \cref{eq:quaternion-map}.
Thus, we can guarantee that
\begin{align}
\begin{aligned}
\mathcal{WA}		&: \admsWA \mapsto \IH \times \IH, \quad (z,\xi) \mapsto (q,p) && \quad \mathrm{(Physical\ Projection)}
\\
\mathcal{WA}^{-1}	&: \IH \times \IH \mapsto \admsWA, \quad (q,p) \mapsto (z,\xi) && \quad \mathrm{(Lifting\ Map)}
\end{aligned}
\end{align}
where $\mathcal{WA}$ denotes the total WA transformation ($\mathcal{WA}_q $ and $\mathcal{WA}_p $) while $\mathcal{WA}^{-1}$ is its inverse.
This paper calls these transformations \textit{physical projection} and \textit{lifting map}, respectively.
The lifting map $\mathcal{WA}^{-1}$ is discussed further in \cref{sec:lifting-map}.

While \cref{eq:mappingCompositionMomenta} gives the complete map for $\mathcal{WA}$ as a composition of three maps, it is often useful to have an expression that represents $\mathcal{WA}$ as a single map, describing $(q,p)$ directly in terms of $(z,\xi)$.
In particular, such an expression is necessary to write the Hamiltonian in terms of the WA state $(z,\xi)$ explicitly.
\cref{theorem:single-shot-WA} addresses it.

\begin{lemma}
\label{theorem:single-shot-WA}
Expanding the composite mappings in \cref{eq:mappingCompositionMomenta}, we have:
\begin{subequations}
\label{eq:WA-transformation}
\begin{align}
q& = \mathcal{WA}_q(z) = \frac{1}{2} \left[i - (z - i) (z - \conj{z})^{-1} (\conj{z} + i) \right]
\label{eq:mappingSpatial}
\\
p& = \mathcal{WA}_p(z,\xi) 
= - \frac{(z - i)(z - \conj{z})^{-1}(\conj{z} - i)(\conj{z} + i) \xi (z - \conj{z}) (z - i)^{-1}}{\abs{2\mathcal{WA}_q(z) - i} \abs{2\mathcal{WA}_q(z) + i}}
\label{eq:mappingSpatialMomenta}
\end{align}
\end{subequations}
\end{lemma}
\begin{proof}
See \cref{proof:single-shot-WA}.
\end{proof}

\subsection{Lifting Map: Physical to WA coordinates}
\label{sec:lifting-map}
Inverting the mappings in \cref{eq:mappingCompositionMomenta}, we obtain the lifting map from the physical coordinates and momenta to the WA variables:
\begin{subequations}
\label{eq:mappingCompositionMomentaForward}
\begin{align}
&(\q{0}, \p{0}) \mapsto (\q{1}, \p{1}) : &&
\q{1} = i - {2}{(2\q{0} - i)^{-1}}, &&
\p{1} = {\conj{\left(\q{0} - \frac{i}{2}\right)} \p{0} \conj{\left(\q{0} - \frac{i}{2}\right)}}
\label{eq:mobius-forward-a}
\\
&(\q{1}, \p{1}) \mapsto (\q{2}, \p{2}) : &&
\q{2} \in  \{h\in \HH \mid \q{1} = \conj{h}ih \}, &&
\p{2} = - 2i\q{2}\p{1}
\label{eq:hopf-forward-b}
\\
&(\q{2}, \p{2}) \mapsto (\q{3}, \p{3}) : &&
\q{3} = i - {2}{(\q{2} - i)^{-1}},  &&
\p{3} = {\frac{1}{2}{\conj{(\q{2} - i)} \p{2} \conj{(\q{2} - i)}}}
\label{eq:mobius-forward-c}
\end{align}
\end{subequations}
As a direct consequence of \cref{theorem:WAdmissibleCharacterization}, any WA state obtained through the lifting map \cref{eq:mappingCompositionMomentaForward} automatically lies in $\admsWA$.

As clear from \cref{eq:hopf-forward-b}, the lifting map is not unique.
Different choices of $\q{2}$ in \cref{eq:hopf-forward-b} correspond to different WA representations of the same physical state (known as \textit{Hopf fibration} \cite{Roa2016a}).
The KS transformation also has this same non-uniqueness.
Hence, similar to \cite{Oguri2024b}, we parameterize the family of $\q{2}$ that satisfy \cref{eq:hopf-forward-b} via a single parameter.
Since $\q{0}\in\IH$ and \cref{theorem:Hopf_Mobius} imply $\q{1}\in\IH$, let us write $\q{1} = \q{1}_1 i + \q{1}_2 j + \q{1}_3 k$ and define
\begin{align}
\begin{aligned}
R \defeq \sqrt{\frac{\abs{\q{1}} + |\q{1}_1|}{2}},
\quad \text{where}\quad
\abs{\q{1}} = \sqrt{(\q{1}_1)^2 + (\q{1}_2)^2 + (\q{1}_3)^2}.
\end{aligned}
\end{align}
The fiber in \cref{eq:hopf-forward-b} can be parameterized by $\theta \in [0,2\pi)$ as
\begin{align}
\q{2}(\theta) =
\begin{cases}
r_1(\theta)i + r_2(\theta)j + r_3(\theta)k + r_4(\theta), & \q{1}_1 \ge 0,\\
r_2(\theta)i + r_1(\theta)j + r_4(\theta)k + r_3(\theta), & \q{1}_1 < 0,
\end{cases}
\label{eq:hopf-forward-theta}
\end{align}
where
\begin{align}
\begin{aligned}
r_1(\theta) &= R\sin\theta,
\quad
&&r_2(\theta) = \frac{\q{1}_2\sin\theta - \q{1}_3\cos\theta}{2R},
\\
r_3(\theta) &= \frac{\q{1}_2\cos\theta + \q{1}_3\sin\theta}{2R},
\quad
&&r_4(\theta) = -R\cos\theta.
\end{aligned}
\label{eq:hopf-forward-theta-components}
\end{align}
It is straightforward to confirm $\conj{\q{2}(\theta)}i\q{2}(\theta) = \q{1}$ for all $\theta$.

\subsection{Regularity of the WA Transformation}
Let us first discuss the regularity of the WA transformation given in \cref{eq:WA-transformation}.
To that end, first introduce a set $\inftyWA$, defined as:
\begin{align}
\inftyWA = \{z \in\mathbb{H} \mid z - \conj{z} = 0 \}
= \{z \in\mathbb{H} \mid \mathrm{Im}[z] = 0 \}
\label{eq:singular-set}
\end{align}
This set corresponds to the real line of the quaternion space (i.e., the imaginary part is zero).
As stated in \cref{theorem:regularity-WAmap}, the coordinate transformation given by \cref{eq:mappingSpatial} is regular except in the set $\inftyWA$.
\cref{theorem:singularSet} then shows that the set $\inftyWA$ corresponds to the infinity in the original physical coordinates $\R^3$, and hence is not of practical interest.
\begin{lemma}
\label{theorem:regularity-WAmap}
The WA coordinate transformation given in \cref{eq:mappingSpatial} is $C^{\infty}$ on $\mathbb{H}\setminus\inftyWA$.
\end{lemma}
\begin{proof}
Since finite sums and products of $C^{\infty}$ functions yield a $C^{\infty}$ function,
we show that \cref{eq:mappingSpatial} consists of finite sums and products of $C^{\infty}$ functions on the domain $\mathbb{H}\setminus\inftyWA$.
First, the map $z \mapsto q$ consists of multiplications of $(z-i)$, $(z - \conj{z} )^{-1}$, and $(\conj{z} + i)$.
Noting $(z - \conj{z} )^{-1} = (\conj{z} - z) / \abs{z - \conj{z}}^2 $ and $\abs{z - \conj{z}}^2 \neq 0$ on $\mathbb{H}\setminus\inftyWA$, the map $z\mapsto (z - \conj{z})^{-1} $ is $C^\infty$ on $\mathbb{H}\setminus\inftyWA$.
Thus, $z \mapsto q$ also is $C^\infty$ on $\mathbb{H}\setminus\inftyWA$ since quaternion multiplication is bilinear over $\R$ (in other words, if $p=p_1i + p_2j + p_3k + p_4$, $q=q_1i + q_2j + q_3k + q_4$, and $r = pq = r_1i + r_2j + r_3k + r_4$, then each $r_i$ is a polynomial of $p_i$ and $q_i$).
\end{proof}

\begin{lemma}
\label{theorem:singularSet}
Under the transformation given by \cref{eq:mappingSpatial}, the set \cref{eq:singular-set} corresponds to the infinity of the physical space $\R^3$.
\end{lemma}
\begin{proof}
We show that the WA mapping sends $\inftyWA$ to $\infty$ in $\IH $.
First, take any $z$ in $\inftyWA$ and denote it by $z_0$.
Any such $z_0$ satisfies $z_0 - \conj{z_0} = 0$, which implies that $z_0 \in\R \subset \mathbb{H} $.
In other words, the set $\inftyWA $ defined in \cref{eq:singular-set} corresponds to $\{z \in \R \}$. 
Then, calculating the M\"obius map: $z\mapsto\q{2} $ for $z_0 \defeq a \in\R $ necessarily yields $\q{2} = \exp(i\theta)$ for some $\theta \in\R$ because
\begin{align}
\q{2} = i - 2(a - i)^{-1} = \frac{a^2-1}{a^2 + 1} i - \frac{2a}{a^2+1}
\quad \text{and}\quad
\abs{\q{2}} = \abs{\frac{a^2-1}{a^2 + 1} i - \frac{2a}{a^2+1}} = 1
\end{align}
For any $\q{2} = \exp(i\theta) $, the KS map gives:
$\q{1} = \conj{\q{2}} i\ \q{2} = (\cos{\theta} - i\sin{\theta})i(\cos{\theta} + i\sin{\theta}) = i$.
Lastly, the M\"obius map: $\q{1}\mapsto q $ for $\q{1} = i$ yields $\infty$ in $\IH $, and hence $\infty$ in $\R^3$.
\end{proof}

\begin{remark}\label{remark:p-singularity}
While the WA coordinate transformation is regular on $\mathbb{H}\setminus\inftyWA$ as shown in \cref{theorem:regularity-WAmap}, the momenta map $(z, \xi) \mapsto p $ given in \cref{eq:mappingSpatialMomenta} is singular when $z =\pm i $, i.e., when the particle collides with one of the primaries.
Importantly, however, it does not preclude the Hamiltonian and associated dynamical flows from being regularized.
Indeed, \cref{sec:regularized_hamiltonian} introduces a Hamiltonian in the WA space with time regularization and proves their regularity.
Thus, while regularization enables propagating the dynamics through collisions and calculating smooth trajectories of $z(s)$ and $\xi(s)$, we cannot recover the physical momenta, $p$, at collision epochs.
This fact makes physical sense too;
at a collision, the time transformation $\diff t/ \diff s = \eta(q) = 0$, and hence, conversion of the finite $s$-time motion to physical-time velocity is singular.
\end{remark}

\subsection{Vector-Matrix Expression}

To calculate the generalized forces $\hat{\bm{G}}, \hat{G}_0 $ in \cref{eq:canonical-force-regular}, we need to take a partial derivative of $\bm{q}$ with respect to $\bm{z} $ (see \cref{eq:B-matrix}).
On the other hand, the WA transformation given in \cref{eq:WA-transformation} is expressed in quaternionic form, which is not convenient for taking the partial derivatives.

Thus, let us derive the vector-matrix form of the WA transformation \cref{eq:WA-transformation}, expressing $(\bm{q}, \bm{p})$ in terms of $(\bm{z}, \bm{\xi})$.
To that end, define the following terms for compact notation:
\begin{align}
\begin{aligned}
\varrho(z) &\defeq \frac{1}{2}\abs{z-\conj{z}} = \sqrt{z_1^2 + z_2^2 + z_3^2},
\quad
\chi^{(\pm)}(z) \defeq \abs{z\pm i}^2 = (z_1 \pm 1)^2 + z_2^2 + z_3^2 + z_4^2.
\end{aligned}
\label{eq:WA-explicit-aux}
\end{align}
Physically, $\varrho$ represents the distance from the real axis in the quaternion space (hence the distance from the infinity in the physical space; see \cref{theorem:singularSet}) while $\chi^{(\pm)}$ denote the squared distance from one of the two primaries ($\chi^{(+)}$: squared distance from $P_1$ and $\chi^{(-)}$: $P_2$).
These terms enable us to compactly write some key terms, including
\begin{align}
\abs{q\pm\frac{i}{2}}
= \frac{\chi^{(\pm)}}{4\varrho}, \quad
\abs{p}^2
= \frac{16\varrho^4}{\chi^{(-)}\chi^{(+)}}\abs{\xi}^2,
\label{eq:intermediate-terms}
\end{align}
which are easy to derive using the WA map in \cref{eq:WA-transformation} and multiplicativity of the quaternion norm.
Using these terms,
\cref{theorem:WA-transform-explicit} gives the vector-matrix form of \cref{eq:WA-transformation}.

\begin{lemma}
\label{theorem:WA-transform-explicit}
Noting $\bm{q} = \vecIH{(q)}$, $\bm{p} = \vecIH{(p)}$, $\bm{z} = \vecH{(z)}$, and $\bm{\xi} = \vecH{(\xi)}$, the WA transformation \cref{eq:WA-transformation} can be equivalently expressed in vector-matrix form as follows:
\begin{subequations}
\begin{align}
\bm{q} &=
\frac{1}{4\varrho^2}
\begin{bmatrix}
\varrho^2 + z_4^2 + 1 & 0 & 0 & 0 \\
0 & \varrho^2 + z_4^2 - 1 & 2z_4 & 0 \\
0 & -2z_4 & \varrho^2 + z_4^2 - 1 & 0
\end{bmatrix}
\bm{z}
\label{eq:coordinates-explicit}
\\
\bm{p} &= \frac{4}{\chi^{(-)} \chi^{(+)}} M \bm{\xi}
\label{eq:momenta-explicit}
\end{align}
\end{subequations}
where $M\in\R^{3\times 4}$, and its $(i,j)$-elements, $M_{ij}$, are given by
\begin{align}
\begin{aligned}
M_{11} &= \varrho^4 + (z_4^2 + 1)(z_2^2 + z_3^2 - z_1^2),
&
M_{12} &= - 2 z_1 z_2 (z_4^2 + 1),
\\
M_{13} &= - 2 z_1 z_3 (z_4^2 + 1),
&
M_{14} &= 2 z_1 z_4 \varrho^2,
\\
M_{21} &= 2 z_1 \left[z_2(1 - z_4^2) - 2 z_3 z_4 \right],
&
M_{22} &= \varrho^4 + (z_4^2 - 1)(\varrho^2 - 2 z_2^2) - 4 z_2 z_3 z_4,
\\
M_{23} &= 2\left[z_2 z_3(1 - z_4^2) + z_4(\varrho^2 - 2 z_3^2)\right],
&
M_{24} &= 2 \varrho^2 (z_2 z_4 + z_3),
\\
M_{31} &= 2 z_1 \left[z_3(1 - z_4^2) + 2 z_2 z_4 \right],
&
M_{32} &= 2\left[z_2 z_3(1 - z_4^2) - z_4(\varrho^2 - 2 z_2^2)\right],
\\
M_{33} &= \varrho^4 + (z_4^2 - 1)(\varrho^2 - 2 z_3^2) + 4 z_2 z_3 z_4,
&
M_{34} &= 2 \varrho^2 (z_3 z_4 - z_2).
\end{aligned}
\label{eq:momenta-explicit-coeff}
\end{align}
\end{lemma}
\begin{proof}
See \cref{proof:WA-transform-explicit}.
\end{proof}

\section{Regularized Hamiltonian}
\label{sec:regularized_hamiltonian}

Now, let us calculate the Hamiltonian under the WA transformation combined with the time transformation.
We can express the time-transformed ER3BP Hamiltonian \cref{eq:timeHamiltonian} using the quaternionic variables $(q_0,p_0,q,p)$ as:
\begin{align}
\begin{aligned}
\hat{H}_e(q_0,p_0,q,p) = &\
\eta(q) \left(\frac{1}{2}\abs{p}^2 + \omega_z(q_0)(q_2 p_1 - q_1p_2) - l_0 q_1 + p_0\right) + \eta(q)V(q) 
\label{eq:Hamiltonian-original}
\end{aligned}
\end{align}
where recall that $q_0 = t$, $q = q_1 i + q_2 j + q_3 k $, and $p = p_1 i + p_2 j + p_3 k $.
As discussed in \cref{sec:time-transform}, \cref{eq:Hamiltonian-original} still has singularities before applying the WA transformation to $(q,p)$.

Substituting the WA mapping given in \cref{theorem:single-shot-WA} into \cref{eq:Hamiltonian-original} yields the Hamiltonian in terms of the WA state $(z,\xi)$, denoted by $K(q_0, p_0, z, \xi)$, as summarized in \cref{theorem:regularHamiltonian-explicit}.

\begin{theorem}\label{theorem:regularHamiltonian-explicit}
Expressing \cref{eq:Hamiltonian-original} in terms of the WA state $(z,\xi)$ yields the Hamiltonian in the following form:
\begin{align}
\begin{aligned}
K(q_0, p_0, z, \xi) =&
\underbrace{\frac{1}{2}\varrho^2|\xi|^2}
_{\substack{\mathrm{regularized}\\\mathrm{kinetic\ energy}}}
+
\underbrace{\eta_{\mathrm{WA}}(p_0-l_0q_1)}
_{\substack{\mathrm{time-scaled}\\\mathrm{extended\ energy}}}
+
\underbrace{\omega_zJ_{\mathrm{WA}}}
_{\substack{\mathrm{rotation}\\\mathrm{term}}}
+
\underbrace{U_{\mathrm{WA}}}
_{\substack{\mathrm{two-primary}\\\mathrm{potential}}}
\label{eq:regularHamiltonian}
\end{aligned}
\end{align}
where 
$\eta_{\mathrm{WA}} \defeq \eta \circ \mathcal{WA}_q $,
$J_{\mathrm{WA}} \defeq [\eta(q) (q_2 p_1 - q_1p_2)] \circ \mathcal{WA} $,
and
$U_{\mathrm{WA}} \defeq  (\eta V) \circ  \mathcal{WA}_q $, given by
\begin{align}
\begin{aligned}
\eta_{\mathrm{WA}}(z) & = \frac{\chi^{(-)}(z)\chi^{(+)}(z)}{16\varrho(z)^2},
\\
J_{\mathrm{WA}}(z,\xi) & =
-\frac{1}{32\varrho(z)^2} \Ree\left\{k\left[z(\bar z-i)(\bar z+i)-i\chi^{(+)}(z)\right]\xi(z-\bar z)(\bar z+i)\right\},
\\
U_{\mathrm{WA}}(z) & = - \frac{(1-\mu)\chi^{(-)}(z)+\mu\chi^{(+)}(z)}{4\varrho(z)}
\end{aligned}
\label{eq:physical-terms-in-WA}
\end{align}
with $\varrho(\cdot)$ and $\chi^{(\pm)}(\cdot)$ defined in \cref{eq:WA-explicit-aux}.
\end{theorem}
\begin{proof}
See \cref{proof:regularHamiltonian-explicit}.
\end{proof}

\subsection{Regularity of the Hamiltonian}

Now that the ER3BP Hamiltonian in $(z,\xi)$ is explicitly given in \cref{eq:regularHamiltonian}, we are ready to formally confirm its regularity.
Denote an interval of $q_0(=t)$ by $\mathcal I$ and its $s$-counterpart by $\mathcal{I}_s $ (where $q_0$ is parameterized by $s$):
\begin{align}
\mathcal{I} \defeq \{q_0 \mid q_0\in[t_0, t_f]  \},
\quad
\mathcal{I}_s \defeq \{s \mid q_0(s) \in \mathcal{I} \} .
\label{eq:interval}
\end{align}
\cref{theorem:regularHamiltonianRegularity} analyzes the regularity of the Hamiltonian given by \cref{eq:regularHamiltonian} on an arbitrary interval $\mathcal{I}$, which clarifies that the Hamiltonian is $C^{\infty}$ except on a set $\inftyWA$ defined in \cref{eq:singular-set}.

\begin{theorem}
\label{theorem:regularHamiltonianRegularity}
The Hamiltonian $K$ given by \cref{eq:regularHamiltonian} is a $C^{\infty}$ function of $(q_0,p_0, z ,\xi)$ on
\begin{align}
\mathcal I\times \R  \times (\mathbb{H}\setminus\inftyWA) \times \mathbb{H},
\label{eq:regular-domain}
\end{align}
where $\inftyWA$ is defined in \cref{eq:singular-set}.
\end{theorem}
\begin{proof}
We show that \cref{eq:regularHamiltonian} is finite sums and products of $C^{\infty}$ functions on the domain defined by \cref{eq:regular-domain}.
First, the following are clear:
$\varrho$ and $\chi^{(\pm)}$ are all $C^{\infty}$ in $z$, and $\abs{z-\conj{z}}^2 > 0$ on $\mathbb{H}\setminus\inftyWA$, $\varrho > 0$.
Thus, $\eta_{\mathrm{WA}} $ and $U_{\mathrm{WA}} $ are both $C^{\infty}$ on the domain \cref{eq:regular-domain}.
It is also obvious that $\abs{\xi}^2$ is $C^{\infty}$ in $\xi$.

Then, it remains to show that $q_1$, $p_0$, $\omega_z(q_0)$, and $J_{\mathrm{WA}} $ are all regular.
As shown in \cref{theorem:regularity-WAmap}, the map $z \mapsto q$ given in \cref{eq:mappingSpatial} is $C^\infty$ on $\mathbb{H}\setminus\inftyWA$, and therefore its first coordinate $q_1$ is also $C^\infty$.
The coordinate $p_0$ enters linearly in \cref{eq:regularHamiltonian}.
$\omega_z(q_0)$ is $C^\infty$ on an arbitrary $\mathcal I$ for the following reason:
\cref{eq:inertial-position} and the relation $\omega_z(t) = \sqrt{1 + e\cos{(\nu(t))}}$ imply that $\omega_z$ is a $C^\infty$ function on $\mathcal I$ because $\nu(t)$ is smooth and $1 + e\cos(\nu(t)) > 0$ for elliptic motion with $0 \le e < 1$.
$J_{\mathrm{WA}} $ is also $C^{\infty}$ in $(z_i,\xi_i)$ because it consists of finite sums and multiplications of quaternions and $\chi^{(\pm)}$, as well as of division by $\varrho(>0)$. 

Consequently, \cref{eq:regularHamiltonian} is $C^\infty$ on the domain \cref{eq:regular-domain}, completing the proof.
\end{proof}

\cref{theorem:regularHamiltonianRegularity} shows that the only possible loss of regularity in \cref{eq:regularHamiltonian} occurs on the set $\inftyWA$.
However, as shown in \cref{theorem:singularSet}, \cref{eq:regularHamiltonian} is singular only when the particle goes to infinity, which is not of practical interest.
Hence, \cref{theorem:regularHamiltonianRegularity,theorem:singularSet} together conclude that the transformed Hamiltonian \cref{eq:regularHamiltonian} is regular for any trajectories of practical interest.

\subsection{Hamilton's Equation}

Let us now derive the equations of motion in the WA space and analyze their regularity.
For notational clarity, let $\bm{b}_{\mathrm{WA}}: \R \times \R^4 \times \R^4\mapsto \R^3 $ represent $\bm{b}_{\mathrm{WA}} \defeq \bm{b} \circ \vecIH \circ \mathcal{WA} \circ \quat $;
explicitly,
\begin{align}
\bm{b}_{\mathrm{WA}}(q_0, \bm{z}, \bm{\xi}) \defeq \bm{b}\left(q_0, \vecIH[\mathcal{WA}_q(\quat{\,\bm z})], \vecIH[\mathcal{WA}_p(\quat{\,\bm z}, \quat{\,\bm \xi})]\right)
\end{align}
where recall \cref{eq:WA-transformation,eq:quaternion-map,eq:b-vector} for $\vecIH(\cdot)$ and $\quat{(\cdot)}$, $\bm{b}(\cdot)$, and the WA mapping, respectively.
Applying the Hamilton's equation to $K(\cdot)$ in \cref{eq:regularHamiltonian} yields
\begin{align}
\begin{aligned}
\frac{\diff \bm{z}}{\diff s} 
=& \frac{\partial K}{\partial \bm{\xi}}
,\qquad
\frac{\diff \bm{\xi}}{\diff s} 
= - \frac{\partial K}{\partial \bm{z}} + \hat{\bm{G}}
\qquad
\frac{\diff q_0}{\diff s} 
= \frac{\partial K}{\partial p_0}
,\qquad
\frac{\diff p_0}{\diff s} 
= - \frac{\partial K}{\partial q_0} + \hat{G}_0
\label{eq:EOM-regular}
\end{aligned}
\end{align}
where $\bm{z} = \vecH(z)$ and $\bm{\xi} = \vecH(\xi)$ using the map defined in \cref{eq:quaternion-map}, and $\hat{\bm{G}}$ and $\hat{G}_0$ denote the generalized forces in the WA space due to $\bm{f}\in\R^3 $, given by
\begin{align}
\begin{aligned}
\hat{\bm{G}} 
= \eta_{\mathrm{WA}}(\bm{z})\, [B(\bm{z})]^\top \bm{f} ,\quad
\hat{G}_0  
= - \eta_{\mathrm{WA}}(\bm{z})\, \bm{b}_{\mathrm{WA}}(q_0, \bm{z}, \bm{\xi}) \cdot \bm{f},
\label{eq:canonical-force-regular}
\end{aligned}
\end{align}
where $B: \R^4 \mapsto \R^{3\times 4}$ is defined as:
\begin{align}
B(\bm{z}) \defeq \frac{\partial \bm{q}}{\partial \bm{z}}
\label{eq:B-matrix}
\end{align}
To calculate $B$ in \cref{eq:B-matrix}, a straightforward approach is to use the vector-matrix expression of $\bm{q}$ in terms of $\bm{z}$, provided in \cref{eq:coordinates-explicit}.

Now, \cref{theorem:regular-EOM} formally shows the regularity of the differential equations \cref{eq:EOM-regular}.

\begin{theorem}\label{theorem:regular-EOM}
The differential equations given in \cref{eq:EOM-regular} are regular on the domain \cref{eq:regular-domain}.
\end{theorem}
\begin{proof}
By \cref{theorem:regularHamiltonianRegularity}, the Hamiltonian $K$ is $C^\infty$ on the domain \cref{eq:regular-domain}.
Hence, its partial derivatives appearing in \cref{eq:EOM-regular} are also $C^\infty$ on that domain.
It remains to consider the generalized-force terms $\hat{\bm G}$ and $\hat G_0$.

First, $\eta_{\mathrm{WA}}$ is $C^\infty$ on $\mathbb H\setminus\inftyWA$ by \cref{theorem:regularHamiltonianRegularity}.
Also, $\bm q=\mathcal{WA}_q(\bm z)$ is $C^\infty$ on this set by \cref{theorem:regularity-WAmap}.
Therefore, $B(\bm z)$ is $C^\infty$, and consequently $\hat{\bm G}$ is $C^\infty$ on the domain.

Next, show that $\hat G_0$ is regular.
Using \cref{eq:momenta-explicit,eq:physical-terms-in-WA,eq:b-vector}, it follows that
\begin{align}
\eta_{\mathrm{WA}}\bm b_{\mathrm{WA}}
=
\frac{1}{4\varrho^2}M(\bm z)\bm\xi
+
\eta_{\mathrm{WA}}(z)\omega_z(q_0)[q_2(z) \hat{\bm{i}} -  q_1(z) \hat{\bm{j}}],
\end{align}
which is $C^\infty$ on $\mathbb H\setminus\inftyWA$ including at collisions with one of the primaries, because $M$ is polynomial in $\bm z$ and $\varrho(z)>0$ on the domain;
note that the singularities of $\bm{p}$ at collisions discussed in \cref{remark:p-singularity} are removed by multiplying $\eta_{\mathrm{WA}}$.
Thus, $\hat G_0 = -\eta_{\mathrm{WA}}\bm b_{\mathrm{WA}}\cdot\bm f $ is also $C^\infty$ on the domain \cref{eq:regular-domain}, completing the proof.
\end{proof}

\subsection{Admissibility of the Dynamical Flow}
Let us now verify that the dynamical flow given by \cref{eq:EOM-regular} preserves the admissibility of trajectories $(z(s),\xi(s))$ in the WA space.
Recall from \cref{remark:hopf-momenta,sec:WA-transformation} that, in order for the WA transformation to always return $(q,p)$ in $\IH\times \IH$ (the original physical space), the WA state $(z,\xi)$ must be in the admissible set $\admsWA$ introduced in \cref{eq:WA-admissibility}.
To that end, \cref{theorem:WA-fiber-momentum} first shows that the following quantity
\begin{align}
\Gamma(z,\xi) \defeq \Ree\left[(\conj{z}+i)\xi(\conj{z}-i)\right]
\label{eq:WA-fiber-quantity}
\end{align}
is a constant of motion under the WA-regularized dynamics.

\begin{lemma}
\label{theorem:WA-fiber-momentum}
$\Gamma(z(s),\xi(s))$ defined in \cref{eq:WA-fiber-quantity} is constant along every solution of \cref{eq:EOM-regular} on the domain \cref{eq:regular-domain}.
\end{lemma}
\begin{proof}
First consider a state away from a collision, so that the three steps in \cref{eq:mappingCompositionMomenta} are well-defined.
Let $(\q{2},\p{2})$ denote the intermediate state after the first M\"obius step \cref{eq:mobius-all-a}.
For $\theta\in\R$, consider the simultaneous rotation
\begin{align}
(\q{2},\p{2}) \mapsto (e^{-i\theta}\q{2},e^{-i\theta}\p{2}).
\label{eq:Hopf-fiber-rotation}
\end{align}
Substituting this into \cref{eq:hopf-all-b} while noting that $i$ commutes with $e^{-i\theta}$ yields:
\begin{align}
\overline{(e^{-i\theta}\q{2})}\,i\,(e^{-i\theta}\q{2})
= \conj{\q{2}}i\q{2},
\qquad
\frac{\overline{(e^{-i\theta}\q{2})}\,i\,(e^{-i\theta}\p{2})}{2|e^{-i\theta}\q{2}|^2}
= \frac{\conj{\q{2}}i\p{2}}{2|\q{2}|^2}
\label{eq:Hopf-fiber-invariance}
\end{align}
which implies that the rotation in \cref{eq:Hopf-fiber-rotation} leaves both outputs (coordinates and momenta) of the KS step unchanged.
Consequently, the physical coordinates and momenta obtained after the final M\"obius step \cref{eq:mobius-all-c} are also unchanged.

Also, for each fixed $\theta$, the rotation \cref{eq:Hopf-fiber-rotation} preserves the canonical one-form at the KS step:
\begin{align}
\Ree\left((e^{-i\theta}\p{2})\overline{\diff(e^{-i\theta}\q{2})}\right)
= \Ree\left(e^{-i\theta}\p{2}\overline{\diff\q{2}}e^{i\theta}\right)
= \Ree\left(\p{2}\overline{\diff\q{2}}\right).
\end{align}
implying that \cref{eq:Hopf-fiber-rotation} is a canonical transformation.
To identify the momentum conjugate to the angle $\theta$, we now consider the infinitesimal variation of $\theta$ and evaluate the canonical one-form on the resulting infinitesimal change of $\q{2}$.
Differentiating \cref{eq:Hopf-fiber-rotation} with respect to $\theta$ at $\theta=0$ gives the infinitesimal fiber rotation
$\delta\q{2} = \left.\diff(e^{-i\theta}\q{2})/\diff\theta\right|_{\theta=0} = -i\q{2}$.
Thus, the momentum conjugate to $\theta$, denoted by $h_{\mathrm{fib}}$, is calculated by evaluating the canonical one-form at the KS step, yielding
\begin{align}
h_{\mathrm{fib}} \defeq
\Ree\left(\p{2}\overline{\delta\q{2}}\right)
= \Ree\left(\p{2}\overline{-i\q{2}}\right)
= \Ree\left(\conj{\q{2}}i\p{2}\right)
\label{eq:Hopf-fiber-momentum}
\end{align}

Now, since \cref{eq:mobius-all-a} is equivalent to
$\conj{z}-i = (\conj{z}+i)\conj{\q{2}}i$ and $2\p{2} = (\conj{z}+i)\xi(\conj{z}+i)$ where note that $z = \q{3} $ and $\xi = \p{3} $, substituting these equations into \cref{eq:WA-fiber-quantity} yields
\begin{align}
\begin{aligned}
\Gamma(z,\xi) 
= \Ree\left[(\conj{z}+i)\xi(\conj{z}+i)\conj{\q{2}}i\right]
= 2\Ree\left(\p{2}\conj{\q{2}}i\right)
= 2\Ree\left(\conj{\q{2}}i\p{2}\right)
= 2h_{\mathrm{fib}},
\end{aligned}
\label{eq:WA-fiber-quantity-relation}
\end{align}
This implies that conservation of $h_{\mathrm{fib}}$ is equivalent to conservation of $\Gamma$.

Thus, let us now prove $\Gamma = \mathrm{constant}$ by showing ${\diff h_{\mathrm{fib}}}/{\diff s} = 0$, where
\begin{align}
\frac{\diff h_{\mathrm{fib}}}{\diff s} 
= \left[\frac{\diff h_{\mathrm{fib}}}{\diff s} \right]_{\mathrm{natural}} 
+ \left[\frac{\diff h_{\mathrm{fib}}}{\diff s} \right]_{\mathrm{forced}}
\label{eq:h-conservation}
\end{align}
The two bracketed terms represent unforced (= natural) and forced contributions to the change rate of $h_{\mathrm{fib}}$.
Away from collision, $K$ is obtained by evaluating $\hat H_e$ at $(q_0,p_0,q,p)$.
The fiber rotation \cref{eq:Hopf-fiber-rotation} leaves the resulting $q$ and $p$ unchanged due to \cref{eq:Hopf-fiber-invariance} (while $q_0$ and $p_0$ are untouched).
Thus, it leaves every argument of $\hat H_e$, and hence $K$, unchanged.
Noether's theorem therefore implies that $h_{\mathrm{fib}}$ is conserved along the dynamical flow when the system is unforced, i.e., 
$[\diff h_{\mathrm{fib}} / \diff s]_{\mathrm{natural}} = 0$.
For the external-force term, let $\delta\bm z$ be the WA-coordinate tangent vector corresponding to the infinitesimal fiber rotation $\delta\q{2}=-i\q{2}$, related by the first M\"obius step \cref{eq:mobius-all-a}.
Since \cref{eq:mobius-all-a} is a cotangent lift, the fiber rotation is canonical in the WA variables, and its momentum satisfies $\partial h_{\mathrm{fib}}/\partial\bm\xi=\delta\bm z$.
Also, because the fiber rotation produces no change in the physical coordinate $\bm q$, it satisfies $B\delta\bm z = (\partial \bm{q}/\partial \bm{z}) \delta\bm z=0$.
Consequently,
\begin{align}
\left[\frac{\diff h_{\mathrm{fib}}}{\diff s}\right]_{\mathrm{forced}}
= \frac{\partial h_{\mathrm{fib}}}{\partial\bm\xi}\cdot\hat{\bm G}
= \delta\bm z\cdot\hat{\bm G}
= \eta_{\mathrm{WA}}\bm f^\top(B\,\delta\bm z)
= 0.
\end{align}
Thus, $[\diff h_{\mathrm{fib}} / \diff s]_{\mathrm{natural}} = [\diff h_{\mathrm{fib}} / \diff s]_{\mathrm{forced}} = 0$, which implies $\diff \Gamma / \diff s = 2 (\diff h_{\mathrm{fib}} / \diff s) = 0$, and hence $\Gamma$ is conserved along \cref{eq:EOM-regular} away from collisions.

Finally, extend this result to collision states.
By \cref{theorem:regular-EOM}, the regularized vector field $\diff(z,\xi)/\diff s$ is continuous on \cref{eq:regular-domain}; 
in addition, $\nabla_{(z,\xi)}\Gamma$ is continuous because $\Gamma$ is polynomial.
Thus,
${\diff\Gamma}/{\diff s} = \nabla_{(z,\xi)}\Gamma\cdot{\diff(z,\xi)}/{\diff s}$
is continuous on \cref{eq:regular-domain}.
The collision set $\{z = \pm i\} $ is a measure-zero subset of the domain \cref{eq:regular-domain} and hence has empty interior, implying that every collision state is a limit of collision-free states.
Therefore, the continuity of ${\diff\Gamma}/{\diff s}$, together with $\diff \Gamma / \diff s = 0$ away from collisions, requires that $\diff\Gamma/\diff s=0$ at collision states as well.
\end{proof}

Recognizing that the set defined by $\Gamma = 0$ is equivalent to $\admsWA$ in \cref{eq:WA-admissibility}, it is clear that \cref{theorem:WA-fiber-momentum} establishes the invariance of $\admsWA$.
Building on this result, \cref{theorem:LambdaInvariant} shows that 
the WA-transformed dynamical flow given by \cref{eq:EOM-regular} ensures the WA state vector to remain in the domain $\admsWA$ as long as the following conditions are met.

\begin{condition}\label{condition:initial-condition}
For some $s_0\in\mathcal{I}_s $ (e.g., initial time), the WA state $(z(s_0),\xi(s_0))$ is in $\admsWA$. 
\end{condition}

\begin{condition}\label{condition:infinity}
The $z$-trajectory remains away from the set $\inftyWA$ on $\mathcal{I}_s $.
\end{condition}

These conditions are not restrictive in practice, because
\begin{enumerate}
	\item as shown in \cref{theorem:WAdmissibleCharacterization}, \cref{condition:initial-condition} is guaranteed to be met for any $(z,\xi)$ obtained by transforming arbitrary $(q,p)\in\IH\times \IH $,
	and
	\item 
	as shown in \cref{theorem:singularSet}, \cref{condition:infinity} may be violated only when the trajectory goes to the infinity in the original physical space, which is not of practical interest and can be easily detected during the numerical integration of \cref{eq:EOM-regular}.
\end{enumerate}
In practice, if one detects the $z$-trajectory approaching $\inftyWA$, one can simply halt the numerical integration of \cref{eq:EOM-regular} and declare \texttt{escape} of the particle from the system of interest (e.g., spacecraft leaving cislunar space to interplanetary space).

\begin{theorem}
\label{theorem:LambdaInvariant}
Let $(q_0(s),p_0(s),z(s),\xi(s))$ be the solution of \cref{eq:EOM-regular} on the interval $\mathcal{I}_s$ and satisfy \cref{condition:initial-condition,condition:infinity}.
Then, $(z(s),\xi(s))\in\admsWA,\ \forall s\in \mathcal{I}_s$.
\end{theorem}
\begin{proof}
By \cref{condition:initial-condition,eq:WA-admissibility,eq:WA-fiber-quantity}, $\Gamma(z(s_0),\xi(s_0))=0$.
By \cref{theorem:WA-fiber-momentum}, $\Gamma$ is a constant of motion (even under external forces) and hence remains zero for every $s\in\mathcal I_s$.
By definition of \cref{eq:WA-admissibility}, $\Gamma(z(s),\xi(s))=0,\ \forall s\in\mathcal I_s \iff (z(s),\xi(s))\in\admsWA,\ \forall s\in\mathcal I_s$.
\end{proof}

\subsection{Numerical Implementation}
\label{sec:implementation}

\begin{algorithm}
\caption{Globally-Regularized Spatial ER3BP Propagation}
\label{alg:WA-propagation}
\begin{algorithmic}[1]
\State
Set $s\gets s_0$; \textbf{if} ER3BP \textbf{then} $\nu\gets\nu_0$; ($\nu$ not needed for CR3BP)
\Comment{initialization}

\State 
Compute $(z_0, \xi_0)$ from $t_0, \bm{r}_0, \bm{v}_0, \theta $ using \cref{eq:qp-original,eq:q0p0-IC,eq:mappingCompositionMomentaForward}
\Comment{initial condition}

\State
Compute $(q_0(s), p_0(s), z(s), \xi(s)) $ by integrating \cref{eq:EOM-regular} for $s\in [s_0, s_f] $
\Comment{propagation}
\label{step:propagation}

\If{$t_f$ is given and $q_0 = t_f$ while Step \ref{step:propagation}}
	\label{step:complete}
	\State
	Set $s_f \gets s $;\quad
	$\texttt{status} \gets \texttt{complete}$
	\Comment{propagation completed}
\ElsIf{$t_f$ is given and $s = s_f$ while Step \ref{step:propagation}}
	\label{step:incomplete}
	\State
	Set $\texttt{status} \gets \texttt{incomplete}$
	\Comment{need greater $s_f$ to complete propagation}
\ElsIf{$\varrho(z)$ becomes small while Step \ref{step:propagation}}
	\label{step:escape}
	\Comment{use \cref{eq:WA-explicit-aux} for $\varrho(z)$}
	\State
	Set $s_f \gets s $;\quad
	$\texttt{status} \gets \texttt{escape}$
	\Comment{particle escaping the system}
\EndIf

\State
Compute $(t(s), \bm{r}(s), \bm{v}(s))$ from $(q_0(s), z(s), \xi(s)) $ for $s\in [s_0, s_f] $ using \cref{eq:WA-transformation}

\end{algorithmic}
\end{algorithm}

\cref{alg:WA-propagation} summarizes an example numerical implementation for propagating the globally-regularized spatial ER3BP using the WA transformation.
There are a few points to note.
\begin{enumerate}
	\item 
	Since the independent variable of the regularized EOM \cref{eq:EOM-regular} is $s$ instead of $t$, if one wants to propagate the dynamics from $t_0$ to $t_f$, it is advised to give a sufficiently large $s_f$ so that Step \ref{step:complete} of \cref{alg:WA-propagation} is satisfied.
	If $s_f$ is not sufficiently large, \cref{alg:WA-propagation} returns \texttt{incomplete} status.
	\item 
	Since \cref{eq:EOM-regular} is singular when the particle goes to the infinity of the physical space, one should terminate the propagation as the $z$-trajectory approaches the set $\inftyWA$ defined in \cref{eq:singular-set}, which is equivalent to $\varrho(z)= 0 $ as clear from \cref{eq:WA-explicit-aux}.
	Thus, Step \ref{step:escape} of \cref{alg:WA-propagation} terminates the propagation and returns \texttt{escape} status when $\varrho(z)$ becomes small.
	In practice, one can terminate it when $\varrho(z) \leq \epsilon_{\infty} $ with a small threshold $\epsilon_{\infty}>0$.
	\item 
	While the regularized EOM \cref{eq:EOM-regular} allows one to smoothly propagate $(z,\xi)$ through both of the two singularities, its inverse mapping to $(q,p)$ after propagation needs some attention.
	As discussed in \cref{remark:p-singularity}, the inverse mapping from $z$ to $q$ is well-defined, but the momenta mapping $(z,\xi)\mapsto p$ is undefined when $z$ is exactly at one of the collision states (i.e., when $z = \pm i$).
\end{enumerate}

\subsection{Directions for Further Extensions}

Beyond the ER3BP, all the analyses above directly apply to a slightly broader class of R3BPs.
Specifically, it readily extends to systems where the two primaries are on non-Keplerian orbits due to some external perturbations, provided that they remain in the same fixed co-orbital plane, $\bm{R}_{12}(T)$ is a smooth function of time, and the angular momentum vector $(\bm{R}_{12}\times \dot{\bm{R}}_{12})$ remains non-vanishing.
For instance, it encompasses scenarios where the primaries are perturbed by other gravitational bodies in the same orbital plane.

Moreover, the overall approach presented in this paper will likely remain useful for even broader classes of systems.
For instance, one could extend the presented approach to higher-fidelity models, such as the Hill restricted 4-body problem \cite{Scheeres1998b,Park2025a,Brown2025}, by incorporating time-varying orbital plane of the primaries and perturbations due to the fourth body (without attempting to regularize the fourth body).

Another direction of extension includes considering other global regularization coordinates.
The unifying framework presented in this paper should allow one to replace the KS step \cref{eq:hopf-all-b} with a different regularizing transformation in the quaternionic space.

\section{Numerical Examples}
\label{sec:examples}

This section demonstrates the validity and effectiveness of the presented formulation for global regularization of the spatial ER3BP under external forces.
Specifically, this paper considers the Earth-Moon system perturbed by SRP.
For numerical implementation, this study uses \cref{alg:WA-propagation} to propagate the regularized EOM, which implements the analytical partial derivatives of $K$ in \cref{eq:EOM-regular} calculated by a symbolic manipulator.
This numerical study uses Matlab's \texttt{ode89} function to numerically solve the ODE with the relative and absolute tolerances $1\times10^{-12}$;
\texttt{ode89} is an adaptive-step ODE solver that implements the Runge-Kutta 9(8) method in \cite{Verner2010}.
This study uses the \texttt{event} function capability of \texttt{ode89} to implement the escape condition discussed in \cref{sec:implementation,alg:WA-propagation}.

\subsection{Physical System Definition}
\label{sec:physical_system_definition}
\cref{t:param} lists relevant parameters to model the Earth-Moon system.
For the CR3BP, only $\mu, a, n$ are relevant since $e=0$ by definition, and $\omega_z(t) = n$ and $l(t) = a$ are both constant;
$\omega_z(t)$ is used within the dynamics propagation while $l(t)$ is necessary to convert numerical results into the original units.
For the ER3BP,
one can calculate $\omega_z(t)$ from $e$ and $\nu(t)$, where $\nu(t)$ can be obtained by either 
(i) solving the Kepler's equation using $e$ and the mean anomaly $M(t) = M(t=t_0) + n(t - t_0) $,
or 
(ii) integrating 
$\diff \nu / \diff t = \sqrt{1 + e\cos\nu} $ for Cartesian propagation, 
or
${\diff \nu}/{\diff s} = ({\diff \nu}/{\diff t})({\diff t}/{\diff s}) = \eta_{\mathrm{WA}}\,\sqrt{1 + e\cos\nu}$ for the regularized dynamics.
Likewise, one can calculate $l(t)$ as $l(t) = {a(1 - e^2)}/{(1 + e\cos{\nu(t)})} $ to recover numerical results in the original units.

\begin{table}[tb]
\centering
\caption{Physical parameters used for the Earth-Moon system.}
\label{t:param}
{\small
\begin{tabular}{c|c|c|c|c|c}
\hline \hline
		& mass ratio 					& eccentricity 			& semi-major axis 			& mean motion 			& true anomaly \\
\hline
Symbol 	& $\mu$ [-]		& $e$ [-]	& $a$ [km]	& $n\ [\mathrm{rad/s}] $ 		& $\nu(t=t_0) $ [rad] \\
\hline
Value 	& $1.215059\times 10^{-2}$ 		& $0.0549$ 				& $3.844\times 10^5$ 		& $2.6617\times10^{-6}$ & $0$  \\
\hline \hline
\end{tabular}
}
\end{table}

Let us now discuss the relevant parameters to model SRP as the external acceleration $\bm{f}$.
Assuming the standard cannon-ball model \cite{Vallado2001a}, we model SRP acceleration as $\bm{f}_{\mathrm{SRP}} = f_{\mathrm{SRP}} \hat{\bm{s}} $, where $\hat{\bm{s}}$ is the unit vector aligned with the Sun-particle direction, while $f_{\mathrm{SRP}} = 3\times 10^{-5}$ in the non-dimensional unit;
this roughly corresponds to the magnitude of SRP acceleration acting on a small satellite (e.g., 12-unit CubeSat) of area-to-mass ratio $0.012\ \mathrm{m^2/kg}$ with reflectivity ratio $0.5$ at 1 AU.
To calculate $\hat{\bm{s}} $, the Earth-Moon barycenter is assumed to be orbiting the Sun with the heliocentric period $365.2$ days, with inclination of $5.145$ deg relative to the Earth-Moon orbit plane, and with the heliocentric initial true anomaly $60 $ deg.

\begin{figure}[tb]
\centering 
\includegraphics[width=\linewidth]{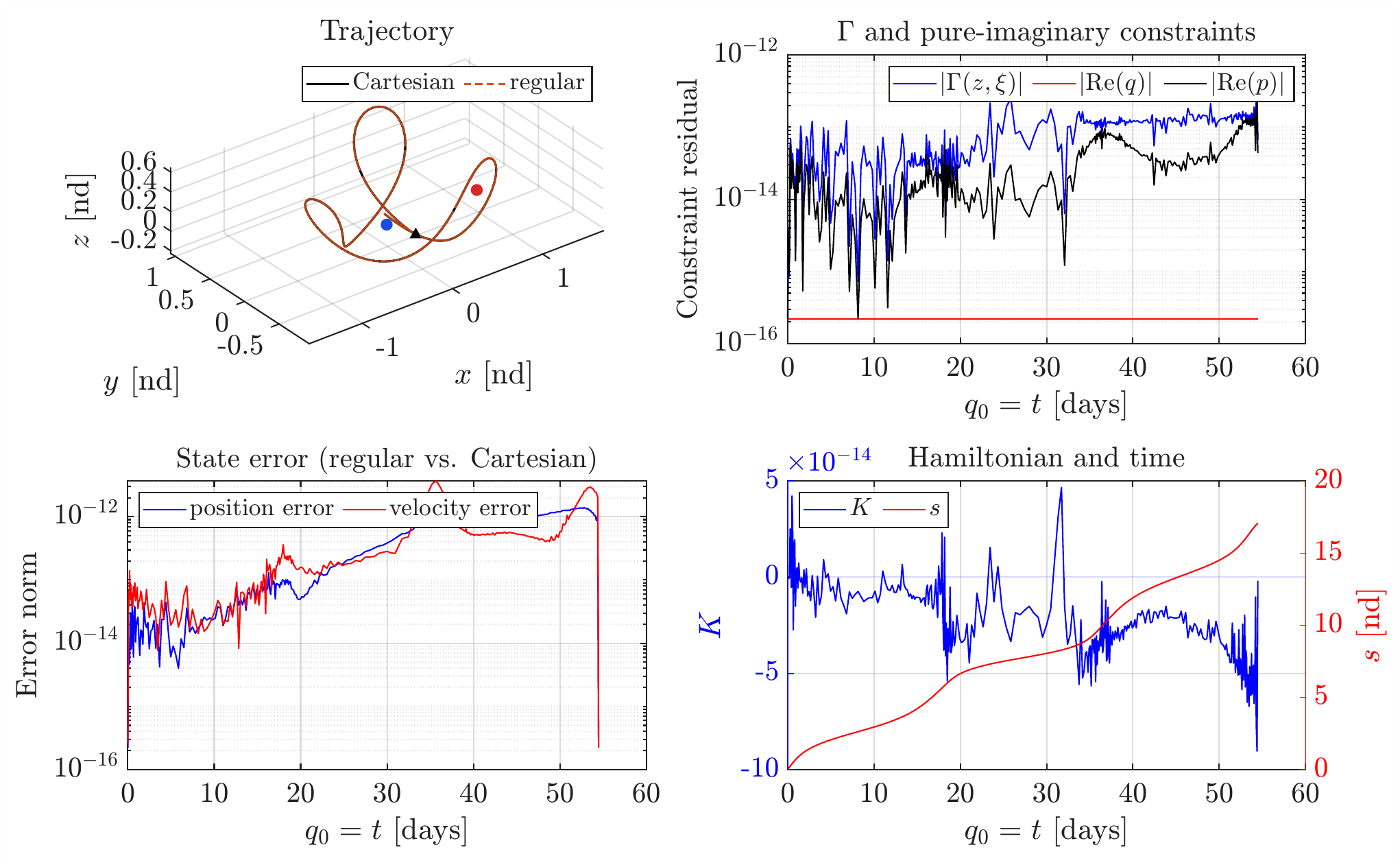}
	\caption{\label{f:validation} Validation simulation using a spatial synodic resonant orbit traveling across the Earth-Moon system modeled as a SRP-perturbed Earth-Moon ER3BP (hence not periodic in this system), confirming the consistency of the regularized and Cartesian propagation and the conservation of constants.}
\end{figure}

\subsection{Validation}
\label{sec:validation}
This section validates the presented formulation using representative trajectories that do \textit{not} approach any of the primaries.
It is because, if a trajectory encounters singularity, Cartesian propagation yields inaccurate results (as observed in \cref{sec:trajectories_with_collisions}), which makes it not straightforward to test the consistency between the regularized and Cartesian dynamics.
This section compares numerical solutions to the regularized EOM \cref{eq:EOM-regular} against those from the conventional Cartesian propagation.
Since both EOMs describe the same dynamics, they must match when the trajectory does not approach either of the primaries.

\cref{f:validation} summarizes one of the validation simulations conducted.
This simulation considers a spatial synodic resonant orbit that travels across the Earth-Moon system, which has various utilities in space operations \cite{Gupta2024}.
The initial condition is listed in \cref{t:initial-cartesian,t:initial-regular}, where \cref{t:initial-cartesian} provides the initial conditions in the Cartesian space while \cref{t:initial-regular} lists those in the regularization space.
The initial regularization state is obtained by applying the lifting map given in \cref{sec:lifting-map} with $\theta = \pi/2$;
numerical simulations with different $\theta$ values confirmed the same trend.
The rest of numerical simulations in this paper also use $\theta = \pi/2$.

\begin{table}[tb]
\centering
\caption{Initial conditions in $(\bm{r}, \bm{v})\ \mathrm{[nd]} $ for the simulations shown in \cref{f:validation,f:CR3BP-P2,f:ER3BP-SRP-P1,f:ER3BP-SRP-P2}.}
\label{t:initial-cartesian}
{\small
\begin{tabular}{@{}c|ccc|ccc@{}}
\hline \hline
Model (Fig. \#) & $r_1$ & $r_2$ & $r_3$ & $v_1$ & $v_2$ & $v_3$ \\
\hline
ER3BP + SRP (\cref{f:validation}) & $0.310545$ & $0$ & $-0.215848$ & $0$ & $1.625235$ & $0$ \\
\hline
CR3BP (\cref{f:CR3BP-P2}) & $0.731656$ & $0.282842$ & $0$ & $0.0276066$ & $0.324744$ & $0$ \\
\hline
ER3BP + SRP (\cref{f:ER3BP-SRP-P1}) & $2.09747$ & $1.26590$ & $0.343864$ & $1.30823$ & $-2.10743$ & $0.0503013$ \\
\hline
ER3BP + SRP (\cref{f:ER3BP-SRP-P2}) & $1.02511$ & $-0.137996$ & $0$ & $-0.177550$ & $-0.153317$ & $0$ \\
\hline \hline
\end{tabular}
}
\end{table}

\begin{table}[tb]
\centering
\caption{Initial conditions in the regularization state $(\bm{z}, \bm{\xi}, q_0 = 0, p_0)\ \mathrm{[nd]} $ for the simulations shown in \cref{f:validation,f:CR3BP-P2,f:ER3BP-SRP-P1,f:ER3BP-SRP-P2}.}
\label{t:initial-regular}
{\small
\setlength{\tabcolsep}{4pt}
\begin{tabular}{@{}c|cccc|cccc|c@{}}
\hline \hline
Case & $z_1$ & $z_2$ & $z_3$ & $z_4$ & $w_1$ & $w_2$ & $w_3$ & $w_4$ & $p_0$ \\
\hline
\cref{f:validation} & $-0.2935$ & $0.8612$ & $0$ & $0.4150$ & $-0.2204$ & $0.6465$ & $0.3617$ & $0.3115$ & $1.171$ \\
\hline
\cref{f:CR3BP-P2} & $0.7450$ & $1.637$ & $0$ & $0$ & $-0.04365$ & $0.1859$ & $0$ & $0$ & $1.409$ \\
\hline
\cref{f:ER3BP-SRP-P1} & $0.1053$ & $-0.04895$ & $-0.07355$ & $-0.3692$ & $6.575$ & $3.011$ & $-0.1527$ & $0.7772$ & $0.2544$ \\
\hline
\cref{f:ER3BP-SRP-P2} & $0.4928$ & $0.2338$ & $0$ & $0$ & $0.2531$ & $-0.07620$ & $0$ & $0$ & $1.410$ \\
\hline \hline
\end{tabular}
}
\end{table}

The top-left panel of \cref{f:validation} compares the 3-D physical trajectories obtained by the Cartesian propagation and by converting the result from regularized propagation to the Cartesian state.
It provides a visual confirmation that both results are consistent with each other.
Note that the chosen orbit is periodic in the unperturbed CR3BP, and hence the trajectory in this figure is not periodic due to the different dynamical model.

The bottom-left panel quantitatively confirms the consistency between the regularized and Cartesian EOMs.
It plots the time histories of the position and velocity differences between the two trajectories, indicating their agreement to the numerical integration tolerance ($1\times10^{-12}$) in the non-dimensional unit.

The top-right panel verifies that the regularized trajectory respects the relevant constraints, including $\Gamma = 0$ and the pure-imaginary constraints on $(q,p)$, i.e., $\Ree{(q)} = \Ree{(p)} = 0$, both within numerical noise.
As stated in \cref{theorem:WA-fiber-momentum}, $\Gamma$ is a constant of motion associated with the fiber rotation.
The time history of $(q,p)$ is obtained by applying the WA transformation \cref{eq:WA-transformation} to the propagated regularization state $(z,\xi)$.
$\Ree{(q)} = \Ree{(p)} = 0$ ensures $(q,p)\in\IH \times \IH $, and hence physically meaningful trajectories.
This plot serves as a quantitative validation of \cref{theorem:WA-fiber-momentum,theorem:LambdaInvariant}.

The bottom-right panel visualizes the Hamiltonian drift ($K(t) - K_0$, where $K_0 = 0$ by construction; see \cref{eq:regularHamiltonian} for $K$) and the evolution of the fictitious time $s$ as a function of the physical time.
Even though the system considered here is non-conservative (ER3BP + SRP), $K(t)$ should remain constant up to numerical integration error, and this figure demonstrates it.
For $s$, we can observe $s$ is monotonically increasing in $t$ as expected, and that the growth of $s$ becomes faster when the particle approaches any of the primaries.

\subsection{Trajectories with Collisions}
\label{sec:trajectories_with_collisions}

This section demonstrates the effectiveness of the presented formulation in handling collisions with primaries.
We compare the following key metrics as measures of effectiveness:
\begin{enumerate}
	\item wall-clock runtime required for completing single propagation,
	\item number of independent-variable steps taken in \texttt{ode89} (adaptive-step ODE solver),
	and
	\item smallest step size taken in \texttt{ode89}.
\end{enumerate}
Clearly, smaller values are desirable for the first and second metrics while for the third metric, being able to take greater step sizes is advantageous.

Specifically, this study evaluates these metrics for each of the regular and Cartesian propagators in simulating three different trajectories.
These trajectories feature different initial conditions and force models, summarized in \cref{t:initial-cartesian} (Cartesian) and in \cref{t:initial-regular} (regularized).
These initial conditions correspond to the results presented in \cref{f:CR3BP-P2,f:ER3BP-SRP-P1,f:ER3BP-SRP-P2}, where \cref{f:CR3BP-P2,f:ER3BP-SRP-P2} considers collisions with the Moon while \cref{f:ER3BP-SRP-P1} involves Earth collisions.
For regularized propagation, \cref{f:CR3BP-P2,f:ER3BP-SRP-P2} uses the range of the fictitious time $s \in [0, 90] $ while \cref{f:ER3BP-SRP-P1} uses $s \in [0, 22.5] $.
The physical final times $t_f$ are recorded for each regularized propagation case, and they are used as the $t_f$ in the physical Cartesian propagation.

In all of these examples, the regularized propagator successfully completes the numerical integration of \cref{eq:EOM-regular} through the collisions with the primaries, which quantitatively validates the regularity claims of the Hamiltonian and EOM, \cref{theorem:regularHamiltonianRegularity,theorem:regular-EOM}.
In addition, the trajectory plots going through the primaries in the position space demonstrates the regularity of regular-to-Cartesian coordinate transformation as stated in \cref{theorem:regularity-WAmap}.

\begin{figure}[tb]
\centering \subfigure[\label{f:CR3BP-P2-reg} Regularized propagation, $x$-$y$ projection (shaded region: forbidden region).]
{\includegraphics[width=0.49\linewidth]{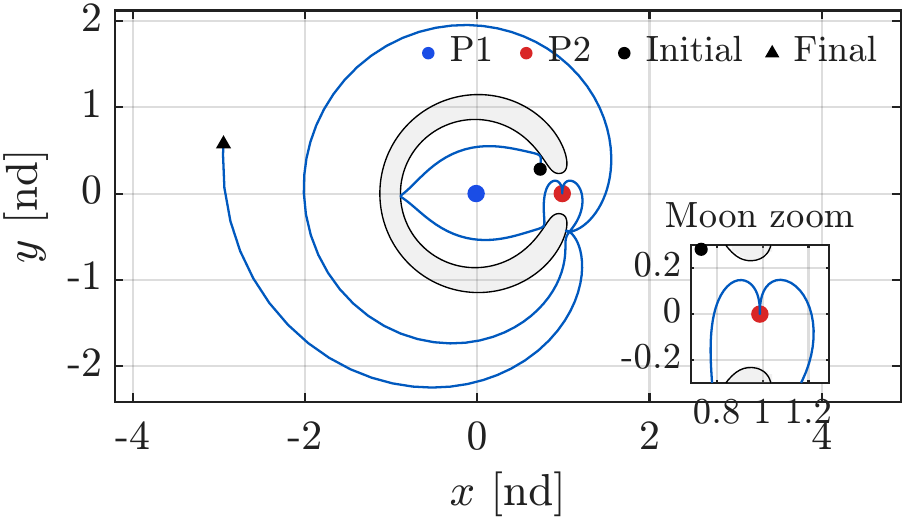}}
\centering \subfigure[\label{f:CR3BP-P2-Car} Cartesian propagation, $x$-$y$ projection.]
{\includegraphics[width=0.49\linewidth]{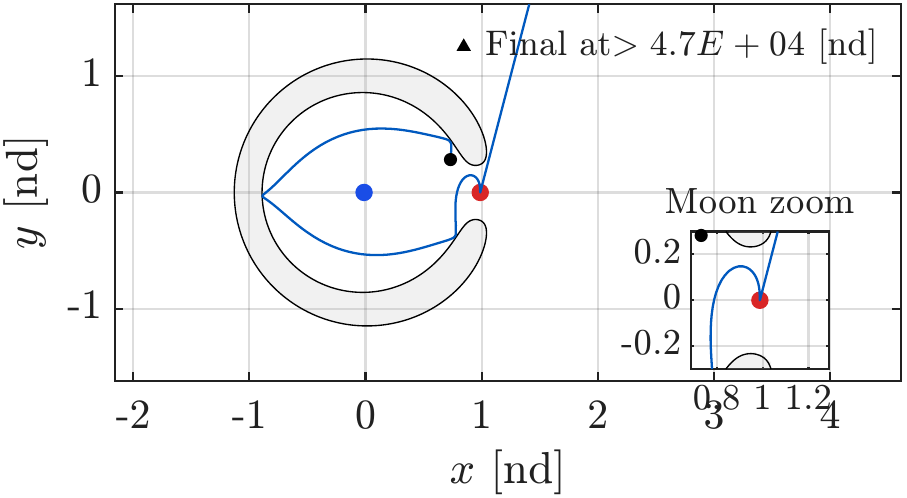}}
\centering \subfigure[\label{f:CR3BP-P2-mom} Inertial angular momentum about each primary.]
{\includegraphics[width=0.49\linewidth]{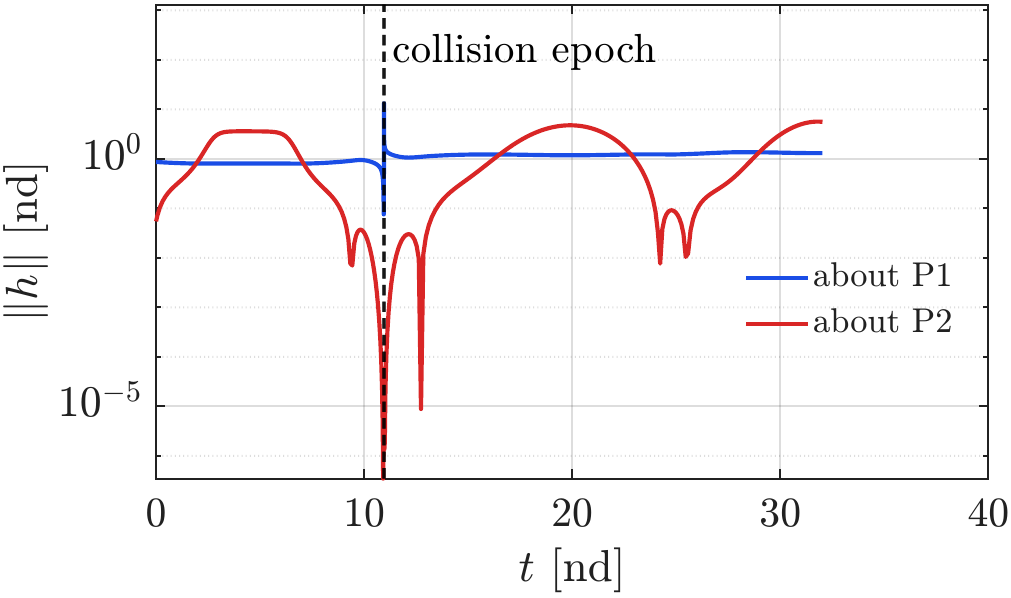}}
\centering \subfigure[\label{f:CR3BP-P2-step} Independent-variable step size comparison.]
{\includegraphics[width=0.49\linewidth]{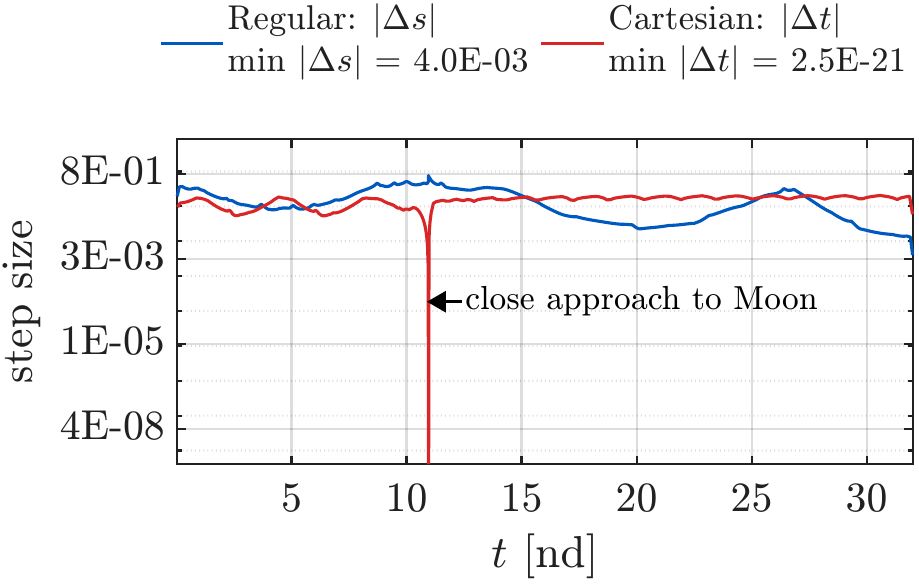}}
	\caption{\label{f:CR3BP-P2} Earth-Moon CR3BP, regularized vs.~Cartesian dynamics, with Moon singularity encounter and $t_f \approx 139$ days.
	The regularized formulation successfully propagates the trajectory through the collision while the Cartesian resulted in cislunar escape due to the numerical error, even though the initial conditions are identical.%
	}
\end{figure}

First, \cref{f:CR3BP-P2} shows results for a Moon-collision trajectory in the unperturbed CR3BP.
\cref{f:CR3BP-P2-reg,f:CR3BP-P2-reg} compares the propagation results with regular and Cartesian EOM, respectively, starting from the same physical initial condition listed in \cref{t:initial-cartesian} (of course, the regularized propagator uses the corresponding regularization state $(z,\xi)$ listed in \cref{t:initial-regular}).
The shaded region indicates the forbidden region at the energy level of the chosen trajectory.
Although the initial portion up to the Moon encounter is identical, these figures show completely different trajectories after the encounter.
In \cref{f:CR3BP-P2-Car}, the particle is ``kicked out'' the system after the Moon close (subsurface) approach and reaches the distance greater than $4.7\times10^4$ in the non-dimensional unit ($>1.8\times 10^{10}$ km).
\cref{f:CR3BP-P2-mom} plots the time histories of the angular momentum of the regularized trajectory about each primary, calculated in the inertial frame;
this shows that the angular momentum about the Moon ($P_2$) drops and approaches extremely close to zero near the time of collision, which physically makes sense because such collisions trajectories yield zero osculating angular momentum momentarily.
Such situations make Milankovitch-type orbital elements \cite{Rosengren2014} not applicable.

\cref{f:CR3BP-P2-step} compares the step sizes of the independent variables as determined by the adaptive-step ODE solver with the regular ($\abs{\Delta s}$) and Cartesian EOMs ($\abs{\Delta t}$), plotted as a function of the physical time $t$.
This figure exhibits the rapid reduction of the Cartesian propagator step size when encountering the Moon, shrinking down to $2.5\times 10^{-21}$ in the non-dimensional time unit, which is around $9.5\times 10^{-16} $ sec in the dimensional unit.
Such small steps are within numerical noise and lead to inaccurate numerical integration of ODEs.
On the other hand, the step size of the regularized propagator remains roughly constant throughout the propagation, enabling efficient and reliable numerical solutions for trajectories involving collisions.
The efficiency of the regularized propagator is evident also in \cref{t:comparison}, which compares the three key metrics mentioned above, indicating that the regularization achieved $>1450\times$ speedup compared to the Cartesian propagation.

\begin{table}[tb]
\centering
\caption{Key metric comparison between regularized propagator and Cartesian propagator.}
\label{t:comparison}
{\small
\begin{tabular}{c|c|c|c}
\hline \hline
Model (Fig. \#)	& {runtime [sec]} 		& {\# steps [-]} 			& {smallest step size [nd]} \\
				& {Regular / Cartesian} & {Regular / Cartesian} 	& {Regular / Cartesian}  \\
\hline CR3BP (\cref{f:CR3BP-P2}) 			& 0.097 / 141.2 & $286$ / $4.6\times10^6$ & $4.0\times10^{-3}$ / $2.5\times10^{-21}$ \\
\hline SRP + ER3BP (\cref{f:ER3BP-SRP-P1}) 	& 0.347 / fail & $404$ / fail & $8.3\times10^{-4}$ / $1.7\times10^{-24}$ \\
\hline SRP + ER3BP (\cref{f:ER3BP-SRP-P2}) 	& 0.187 / 413.7 & $255$ / $4.6\times10^6$ & $1.8\times10^{-3}$ / $2.6\times10^{-21}$ \\
\hline \hline
\end{tabular}
}
\end{table}

\begin{figure}[tb]
\centering \subfigure[\label{f:ER3BP-SRP-P1-reg-3D} Regularized propagation, 3-D view.]
{\includegraphics[width=0.49\linewidth]{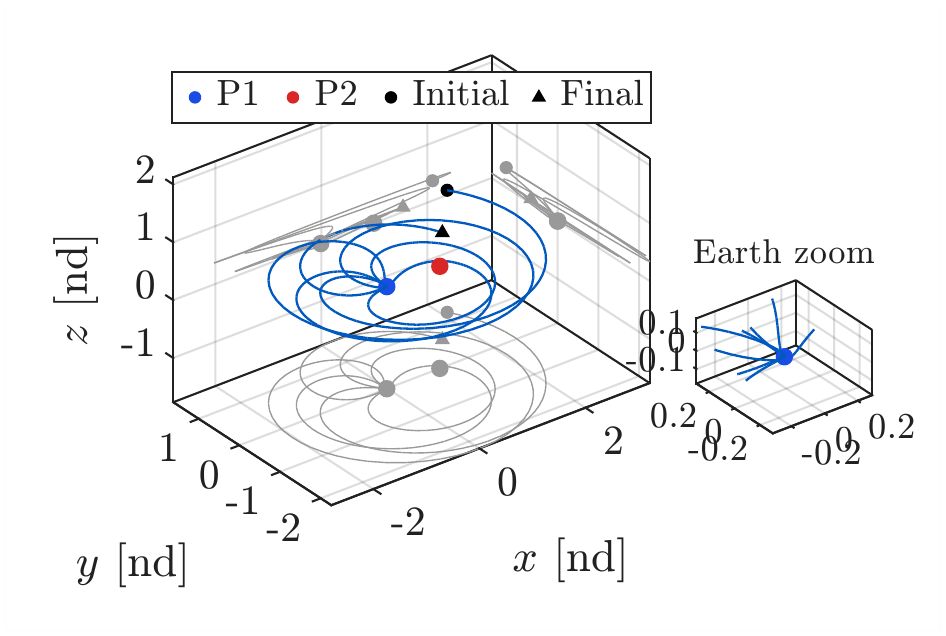}}
\centering \subfigure[\label{f:ER3BP-SRP-P1-Car-3D} Cartesian propagation, 3-D view.]
{\includegraphics[width=0.49\linewidth]{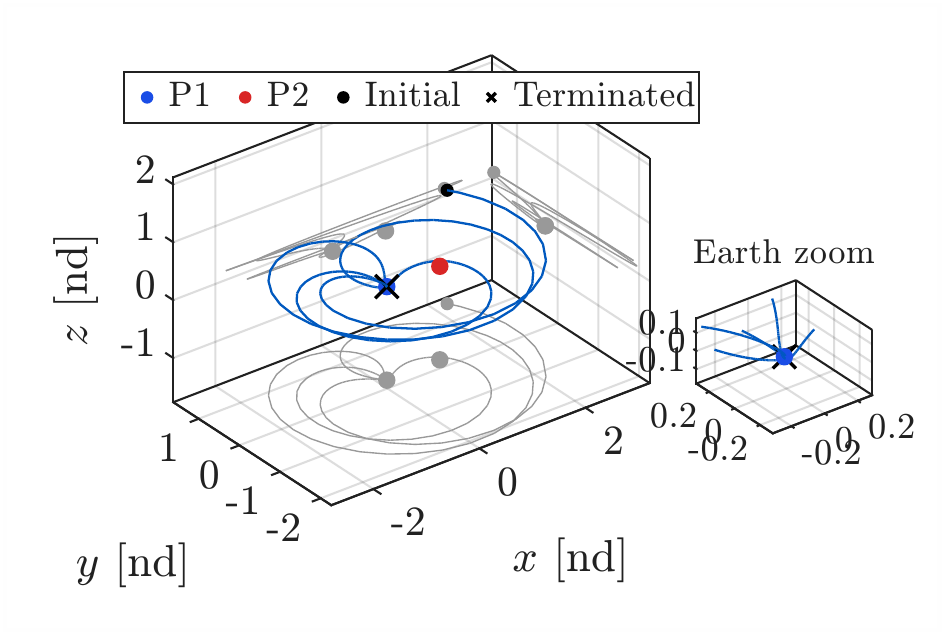}}
\centering \subfigure[\label{f:ER3BP-SRP-P1-step} Independent-variable step size comparison.]
{\includegraphics[width=0.6\linewidth]{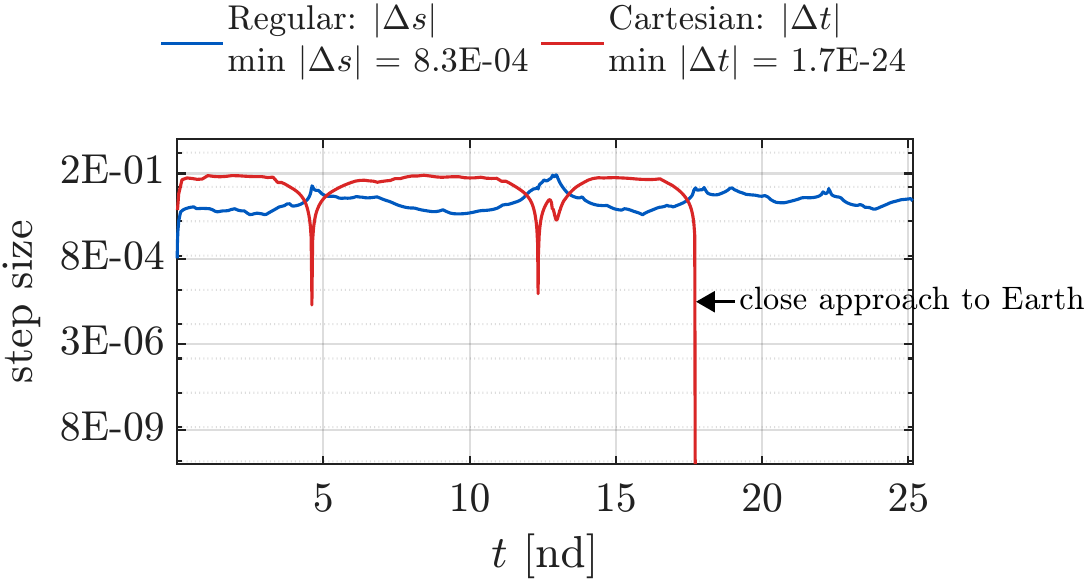}}
	\caption{\label{f:ER3BP-SRP-P1} SRP-perturbed Earth-Moon ER3BP, regularized vs.~Cartesian dynamics, with Earth singularity encounter and $t_f \approx 109$ days.
	The adaptive-step ODE solver (Matlab \texttt{ode89}) failed to complete the Cartesian propagation, not being able to meet the requested tolerance ($10^{-12}$) after reducing the step size below $2\times{10^{-24}}$.%
	}
\end{figure}

Next, \cref{f:ER3BP-SRP-P1} shows an example that involves collisions with Earth in the ER3BP perturbed by SRP (see \cref{sec:physical_system_definition} for SRP modeling).
As can be seen in \cref{t:initial-cartesian} and the $y$-$z$ projections in \cref{f:ER3BP-SRP-P1-reg-3D,f:ER3BP-SRP-P1-Car-3D}, this trajectory has non-zero $z$ component, and yet \cref{f:ER3BP-SRP-P1-reg-3D} demonstrates the successful removal of spatial collision singularities.
On the other hand, \cref{f:ER3BP-SRP-P1-Car-3D,f:ER3BP-SRP-P1-step} indicate that \texttt{ode89} failed to complete the Cartesian propagation for the requested tolerance, where the step size $\Delta t$ dropped below $2\times 10^{-24}$.
Similar to \cref{f:CR3BP-P2}, the regularized propagator step size $\Delta s$ remains about the same throughout the propagation, confirming the stable numerical integration.

\begin{figure}[tb]
\centering \subfigure[\label{f:ER3BP-SRP-P2-reg} Regularized propagation, $x$-$y$ projection.]
{\includegraphics[width=0.49\linewidth]{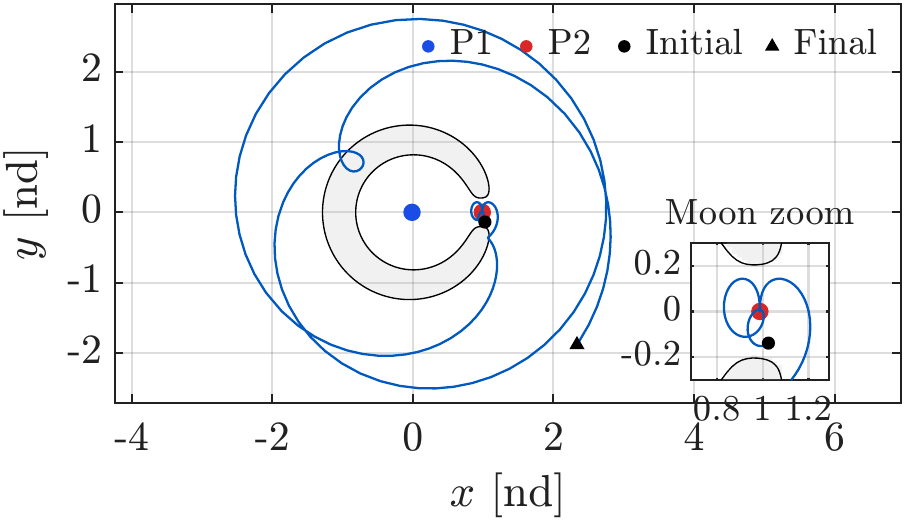}}
\centering \subfigure[\label{f:ER3BP-SRP-P2-Car} Cartesian propagation, $x$-$y$ projection.]
{\includegraphics[width=0.49\linewidth]{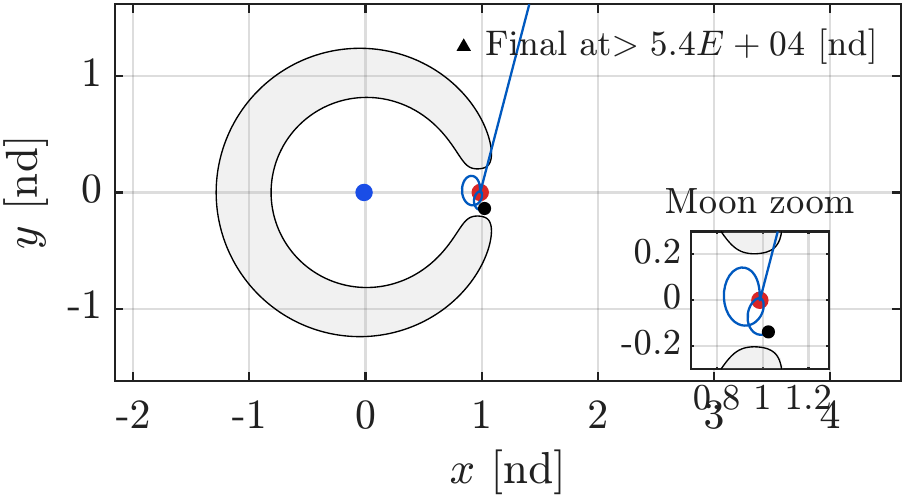}}
	\caption{\label{f:ER3BP-SRP-P2} SRP-perturbed Earth-Moon ER3BP, regularized vs.~Cartesian dynamics, with Moon singularity encounter and $t_f \approx 119$ days.
	Since the system is non-conservative, the particle can enter the shaded region, which is enclosed by the zero-velocity curve at the initial epoch.
	In Cartesian, the particle escape the system after the second Moon encounter, highlighting the inaccuracy caused by the singularity in the Cartesian dynamics.%
	}
\end{figure}

Lastly, \cref{f:ER3BP-SRP-P2} compares trajectories propagated using the regularized and Cartesian EOMs.
Similar to what is observed in \cref{f:CR3BP-P2}, \cref{f:ER3BP-SRP-P2-reg,f:ER3BP-SRP-P2-Car} display completely different trajectories, even though the trajectories are identical in the beginning (see the zoom-in windows).
In the Cartesian propagation, the second Moon encounter kicked out the particle from the system, reaching $\approx2.2\times10^{10}$ km within the same physical time as the regularized propagation.
Note that the shaded region here represents the region enclosed by the zero-velocity curve based on the initial energy;
since the system is non-conservative, this region does not serve as a forbidden region unlike the CR3BP considered in \cref{f:CR3BP-P2}.
As listed in \cref{t:comparison}, the regularized propagator achieved $>2200\times$ speedup compared to the Cartesian counterpart due to the ability to take larger step size in the independent variable throughout the propagation, as evidenced by the difference in the smallest step size.
This result reinforces the observation that the regularized dynamics offer significant computational advantages over the conventional formulation when the trajectory involves close approaches to gravitational bodies in the ER3BP.

\section{Conclusions}
\label{sec:conclusions}
This paper presents a formulation that globally (= simultaneously) regularizes the two singularities in the spatial elliptic restricted three-body problem (ER3BP).
The presented formulation builds on the foundation laid by Waldvogel and Arenstorf in 1960s and provides a unifying framework for global spatial regularization of the ER3BP.
Unlike the existing approaches,
this paper views the regularizing transformation as a composition of three conformal-like mappings in the quaternionic (= hyper-complex) space and leverages mathematical tools from Hamiltonian mechanics to provide a systematic procedure for deriving the globally regularizing transformation.
The Hamiltonian approach reveals the underlying structure of the regularized dynamics, naturally accommodates external acceleration, and offers multiple pathways for further extensions.
This paper proves the regularity of the derived transformation, Hamiltonian, and equations of motion.
It also clarifies that the resulting regularized dynamics are singular only when the particle goes to the infinity of the physical space, which has an important implication in numerical implementation.
Numerical simulations in the Earth-Moon system demonstrates the validity of the mathematical results as well as the numerical advantages of the regularized dynamics compared to the conventional Cartesian representation.

\appendix
\section{Proof of \cref{theorem:single-shot-WA}}
\label{proof:single-shot-WA}
\begin{proof}
Let us first show the coordinate map \cref{eq:mappingSpatial}.
Using \cref{eq:mappingCompositionMomenta}, substitution of \(\q{2}=i-2(z-i)^{-1}\) into \(\q{1}=\conj{\q{2}}i\q{2}\) gives
\begin{align}
(\bar z+i)(\q{1}-i)(z-i)
&=4i+2(z-i)-2(\bar z+i)
=2(z-\bar z).
\label{eq:single-shot-coordinate-identity}
\end{align}
Therefore, $(\q{1}-i)^{-1}=({1}/{2})(z-i)(z-\bar z)^{-1}(\bar z+i)$.
Substitution into \cref{eq:mobius-all-c} yields $q=({1}/{2})\left[i-(z-i)(z-\bar z)^{-1}(\bar z+i)\right]$, which proves \cref{eq:mappingSpatial}.

Next, derive the momenta map.
Combining the three cotangent lifts in \cref{eq:mappingCompositionMomenta} gives
\begin{align}
p =\frac{(\conj{\q{1}}+i)\conj{\q{2}}i\,\p{2}(\conj{\q{1}}+i)}{2\abs{\q{2}}^2}
\label{eq:p-temp}
\end{align}
\cref{eq:mobius-all-a}, \cref{eq:single-shot-coordinate-identity}, and the inverse of \cref{eq:mobius-all-c} respectively imply
\begin{align}
\begin{aligned}
\conj{\q{2}} &=-(\bar z+i)^{-1}(\bar z-i)i,
\qquad
\conj{\q{1}}+i =-2(\bar z+i)^{-1}(z-\bar z)(z-i)^{-1},
\\
\conj{\q{1}}+i & =\frac{\abs{\q{1}-i}^2}{2}(i-2q)
=\frac{\abs{\q{1}-i}^2}{2} (z-i)(z-\bar z)^{-1}(\bar z+i).
\label{eq:single-shot-intermediate-identities}
\end{aligned}
\end{align}
Using \cref{eq:single-shot-intermediate-identities} and \(\p{2}=(\bar z+i)\xi(\bar z+i)/2\) leads to
\begin{align}
(\conj{\q{1}}+i)\conj{\q{2}}i = \frac{\abs{\q{1}-i}^2}{2} (z-i)(z-\bar z)^{-1}(\bar z-i),\quad
\p{2}(\conj{\q{1}}+i) =-(\bar z+i)\xi(z-\bar z)(z-i)^{-1}
\label{eq:single-shot-momentum-identities}
\end{align}
Substituting \cref{eq:single-shot-momentum-identities} into \cref{eq:p-temp} yields
\begin{align}
p &=-\frac{\abs{\q{1}-i}^2}{4\abs{\q{2}}^2} (z-i)(z-\bar z)^{-1}(\bar z-i)(\bar z+i) \xi(z-\bar z)(z-i)^{-1}.
\label{eq:single-shot-momentum-pre}
\end{align}

Finally, the inverse of \cref{eq:mobius-all-c} and \(\abs{\q{1}}=\abs{\q{2}}^2\) give
$\abs{\q{1}-i}=\frac{2}{\abs{2q-i}}$ and $\abs{\q{2}}^2=\abs{\q{1}}=\frac{\abs{2q+i}}{\abs{2q-i}}$, respectively.
It then leads to $\frac{\abs{\q{1}-i}^2}{4\abs{\q{2}}^2}=\frac{1}{\abs{2q-i}\abs{2q+i}}$.
Substituting this relation into \cref{eq:single-shot-momentum-pre} gives \cref{eq:mappingSpatialMomenta}.
\end{proof}

\section{Proof of \cref{theorem:WA-transform-explicit}}
\label{proof:WA-transform-explicit}
\begin{proof}
Since \(z-\bar z\) is pure imaginary, $(z-\bar z)^{-1}= - {(z-\bar z)}/{(4\varrho^2)}$, where \cref{eq:WA-explicit-aux} is used.
Substitution into the coordinate map in \cref{eq:WA-transformation} and direct expansion give
\[
\begin{aligned}
q_1 = \frac{(\varrho^2+z_4^2+1)z_1}{4\varrho^2},\ 
q_2 = \frac{(\varrho^2+z_4^2-1)z_2+2z_4z_3}{4\varrho^2},\ 
q_3 = \frac{(\varrho^2+z_4^2-1)z_3-2z_4z_2}{4\varrho^2},\ 
q_4 =0
\end{aligned}
\]
which gives the coordinate map \cref{eq:coordinates-explicit}.

Now consider the momentum map.
Using \cref{eq:intermediate-terms} together with $(z-i)^{-1}={(\bar z+i)}/{(\chi^{(-)})}$ and $(z-\bar z)^{-1}=-{(z-\bar z)}/{(4\varrho^2)}$ reduces \cref{eq:mappingSpatialMomenta} to
\begin{align}
p
&=
\frac{(z-i)(z-\bar z)(\bar z-i)(\bar z+i)\,
\xi\,(z-\bar z)(\bar z+i)}
{\bigl[\chi^{(-)}\bigr]^2\chi^{(+)}}.
\label{eq:momenta-explicit-reduced}
\end{align}
Expanding the left and right quaternion products in \cref{eq:momenta-explicit-reduced} component-wise and collecting the coefficients of $\bm{\xi}$ gives
\[
\begin{bmatrix}
p_1 & p_2 & p_3
\end{bmatrix}^\top
=\frac{4}{\chi^{(-)}\chi^{(+)}}M(z)\bm\xi,
\]
where the entries of \(M\) are those listed in \cref{eq:momenta-explicit-coeff}.
It remains to verify that $p_4 = \Ree{p} $ vanishes.
Cycling the first two factors of \cref{eq:momenta-explicit-reduced} to the end of the product gives
\begin{align}
\begin{aligned}
\Ree(p)
&=
\frac{1}{[\chi^{(-)}]^2\chi^{(+)}}
\Ree\!\left[(\bar z-i)(\bar z+i)\xi
(z-\bar z)(\bar z+i)(z-i)(z-\bar z)\right]
\\
&= -\frac{4\varrho^2}{\chi^{(-)}\chi^{(+)}}
\Ree\left[(\bar z+i)\xi(\bar z-i)\right] \underset{\cref{eq:WA-admissibility}}{=} 0
\label{eq:p4-verification}
\end{aligned}
\end{align}
where the second equality uses \((\bar z+i)(z-i)=\chi^{(-)}\), \((z-\bar z)^2=-4\varrho^2\), and one final cyclic permutation inside \(\Ree(\cdot)\).
This completes the proof.
\end{proof}

\section{Proof of \cref{theorem:regularHamiltonian-explicit}}
\label{proof:regularHamiltonian-explicit}
\begin{proof}
Using \cref{eq:intermediate-terms} directly yields the first and third terms of \cref{eq:physical-terms-in-WA}:
\begin{align}
&\eta(q) = \abs{q-\frac{i}{2}}\abs{q+\frac{i}{2}} = \frac{\chi^{(-)}\chi^{(+)}}{16\varrho^2} = \eta_{\mathrm{WA}},
\label{eq:eta-WA-derivation}
\\
&\eta(q)V(q) = -(1-\mu)\abs{q-\frac{i}{2}}-\mu\abs{q+\frac{i}{2}} = U_{\mathrm{WA}},
\label{eq:potential-WA-derivation}
\\
&\frac{1}{2}\eta(q)\abs{p}^2
= \frac{1}{2}\varrho^2\abs{\xi}^2 \quad \text{(regularized kinetic energy)}
\label{eq:eta-WA-kinetic}
\end{align}

It remains to obtain the rotation term (the second term of \cref{eq:physical-terms-in-WA}). 
Since $q_2p_1-q_1p_2=\Ree(kqp)$, the momentum map in \cref{eq:WA-transformation} and \cref{eq:eta-WA-derivation} give
\begin{align}
J_{\mathrm{WA}}
= \eta(q)(q_2p_1-q_1p_2) 
= -\frac{1}{4}\Ree\!\left\{
kq(z-i)(z-\bar z)^{-1}(\bar z-i)(\bar z+i)
\xi(z-\bar z)(z-i)^{-1}\right\}
\label{eq:J-WA-pre}
\end{align}
Substituting the coordinate map for $q$ in \cref{eq:J-WA-pre}, and using $(\bar z+i)(z-i)=\chi^{(-)}$ and $(z+i)(\bar z-i)=\chi^{(+)}$, yields
\begin{align}
J_{\mathrm{WA}}
&= -\frac{1}{8\abs{z-\bar z}^2}
\Ree\!\left\{
k\left[z(\bar z-i)(\bar z+i)-i\chi^{(+)}\right]
\xi(z-\bar z)(\bar z+i)\right\}
\nonumber\\
&= -\frac{1}{32\varrho^2}
\Ree\!\left\{
k\left[z(\bar z-i)(\bar z+i)-i\chi^{(+)}\right]
\xi(z-\bar z)(\bar z+i)\right\}
\label{eq:eta-rotation}
\end{align}
Substituting \cref{eq:potential-WA-derivation,eq:eta-WA-kinetic,eq:eta-rotation} into \cref{eq:Hamiltonian-original} gives \cref{eq:regularHamiltonian}.
\end{proof}

\section*{Acknowledgments}
The author acknowledges the use of OpenAI's Codex to assist in simplifying the algebraic derivations of proofs, especially those involving quaternion algebra.
The author also used the same AI tool to translate some papers written in non-English languages in the process of literature review to make the novelty claim precise.

\bibliographystyle{siamplain}
\bibliography{../../../../../../utility/latex/ref/zotero/library}

\end{document}